\documentclass[review,hidelinks,onefignum,onetabnum]{siamart220329}

\usepackage{lipsum}
\usepackage{amsfonts}
\usepackage{graphicx}
\usepackage{epstopdf}
\usepackage{algorithmic}
\ifpdf
  \DeclareGraphicsExtensions{.eps,.pdf,.png,.jpg}
\else
  \DeclareGraphicsExtensions{.eps}
\fi

\newsiamremark{remark}{Remark}
\newsiamremark{hypothesis}{Hypothesis}
\crefname{hypothesis}{Hypothesis}{Hypotheses}
\newsiamthm{claim}{Claim}

\headers{}{H. Zhao and D. M. Tartakovsky}

\title{Model discovery for nonautonomous 
translation-invariant problems\thanks{\textbf{Funding:} {This work was funded in part by the Air Force Office of Scientific Research under grant FA9550-21-1-0381, by the National Science Foundation under award 2100927, by the Office of Advanced Scientific Computing Research (ASCR) within the Department of Energy Office of Science under award number DE-SC0023163, and by the Strategic Environmental Research and Development Program (SERDP)  of the Department of Defense under award RC22-3278.}}}

\author{Hongli Zhao\thanks{Department of Statistics, The University of Chicago, Chicago, IL
  (\email{honglizhaobob@uchicago.edu}).}
\and Daniel M. Tartakovsky\thanks{Department of Energy Science and Engineering, Stanford University, Stanford, CA 
  (\email{tartakovsky@stanford.edu}).
  } 
}

\usepackage{amsopn}

\usepackage{mathtools}

\newcommand{\pphi}{\boldsymbol{\Phi}}

\newcommand{\ssigma}{\boldsymbol{\Sigma}}
\newcommand{\norm}[1]{\left\lVert#1\right\rVert}
\newcommand{\normtwo}[1]{\left\lVert#1\right\rVert}

\ifpdf
\hypersetup{
  pdftitle={An Example Article},
  pdfauthor={D. Doe, P. T. Frank, and J. E. Smith}
}
\fi

\begin{document}

\maketitle

\begin{abstract}
Discovery of mathematical descriptors of physical phenomena from observational and simulated data, as opposed to from the first principles, is a rapidly evolving research area. Two factors, time-dependence of the inputs and hidden translation invariance, are known to complicate this task. To ameliorate these challenges, we combine Lagrangian dynamic mode decomposition with a locally time-invariant approximation of the Koopman operator. The former component of our method yields the best linear estimator of the system's dynamics, while the latter deals with the system's nonlinearity and non-autonomous behavior. We provide theoretical estimators (bounds) of prediction accuracy and perturbation error to guide the selection of both rank truncation and temporal discretization. We demonstrate the performance of our approach on several non-autonomous problems, including two-dimensional Navier-Stokes equations. 
\end{abstract}



\section{Introduction} With the advent of machine learning applications in the engineering sciences, the need for pattern recognition and predictions has become increasingly pronounced in order to assist the study of temporally evolving natural phenomena~\cite{naturescientificai}. Direct-solution methods, which often relies on deep neural networks (DNN) to encode an input-output relationship, are hindered by the high requirement on both quantity and quality of data and thus sensitive to parametric changes of the underlying system~\cite{10.1145/3567591}. On the other hand, equation discovery~\cite{Todorovski2010} supplements partial knowledge with optimal predictions/parameter inference to reproduce the governing laws. Well-known methods belonging to this class include symbolic regression~\cite{doi:10.1126/sciadv.aay2631}, numerical Gaussian processes~\cite{Raissi_2017,raissi2017numerical}, sparse identification of nonlinear dynamics (SINDy)~\cite{doi:10.1073/pnas.1517384113}, physics-informed neural networks (PINN)~\cite{physicsinformednndiscovery} and Kalman filters~\cite{Chui2009}, along with combinations of these strategies to accommodate different physical scenarios or achieve computational improvements~\cite{KIYANI2023116258,doi:10.1137/21M1434477,Kaheman2020,doi:10.1137/21M1458880}. 

In the context of system identification with complete absence of physics, equation-free methods are adopted to reconstruct the observed processes through a purely data-driven surrogate. Instead of prescribing a set of dictionary terms, equation-free methods seek to approximate the flow map/operator that incorporates differential information. Deep neural networks (DNN) and dynamic mode decompositions (DMD) are two prominent classes of methods for operator learning. Including the well-known DeepONet~\cite{Lu_2021}, DNN architectures possess high expressiveness and are capable of serving as nonlinear surrogates of PDE-based models to arbitrary accuracy given sufficient training samples~\cite{Qin_2019,CHEN2022110782}. On the other hand, DMD provides an optimal linear approximation of the model and relies on the Koopman operator to account for nonlinearity~\cite[and references therein]{Kutz2016book,Mezic2021}. In the absence of precise error estimators for DNN surrogates, their performance on any given problem cannot be guaranteed \emph{a priori}. In contrast, being a linear surrogate, DMD is better understood and equipped with theoretical estimates of prediction accuracy, e.g.,~\cite{lu2019predictive}. 
Its various generalizations are designed to handle advection-dominated phenomena~\cite{Lu_2020}, shock waves and discontinuous solutions~\cite{lu2020dynamic}, inhomogeneity of differential operators~\cite{Lu_2021} and a problem's parametric dependence~\cite{lu-2023-drips}.

Physical constraints in the PDE model, such as translation invariance and time-dependent coefficients, pose challenges for both DNNs and DMD. For instance, direct-solution DNNs using soft regularization to enforce advection and mass conservation may lead to ill-conditioning during training~\cite{krishnapriyan2021characterizing}. Operator DNNs have also been observed to yield poor performance when the finite data samples are not representative of global transport phenomena~\cite{Zhu_2023,xu2021how}. Likewise, standard DMD is also not devoid of shortcomings and fails to cope with sharp gradients  \cite{doi:10.1137/130932715,Lu_2020}. Furthermore, its construction is based on the assumption of time homogeneity (e.g., parameters and/or source terms do not vary in time), which is not suitable for nonautonomous problems. 
 
A prime example of the twin challenges to model discovery is advection-diffusion problems which encapsulate conservation of momentum~\cite{SHENG202060,Temam2001}, thermal energy~\cite{transport-phenomena-cite}, and probability mass~\cite{Risken1996}. 
%
%
In the diffusion-dominated and intermediary regimes, these problems have been successfully treated via standard reduced-order basis methods including DMD~\cite{lu2019predictive, Nakao2020} and POD~\cite{GUO201975}.
The advection-dominated regime, characterized by, e.g., high P\'eclet and Reynolds numbers, complicates not only numerical solution of advection-diffusion equations but also discovery of these equations (or corresponding equation-free models) from observations.
Although its convergence properties has been well-studied~\cite{mezicdmdconvergence}, standard DMD yields quantitatively and qualitatively incorrect solutions, spurring the development of Lagrangian DMD \cite{Lu_2020}.


Reduced-order surrogate models of nonautonomous dynamical systems require either an appropriate global spatio-temporal basis or a time-dependent parameterization (e.g. via Fourier spectral expansion)~\cite{mezic-spectral,MEZIC2016690}. Examples of such modifications of the standard DMD include streaming DMD~\cite{Hemati_2014}, time-varying DMD~\cite{zhang2017online}, and more generally, nonautonomous Koopman operators for (quasi)periodic time-dependent inputs~\cite{Mezic2016}. We build upon these developments to construct a DMD framework for translation-invariant (e.g., advection-dominated) problems with time-dependent inputs. Our approach is to reformulate a governing partial differential equation in the Lagrangian frame of reference and to deploy a piece-wise constant approximation of the nonautonomous Koopman operator in the resulting Lagrangian DMD~\cite{Lu_2020}. 

In section~\ref{sec:problem-formulation}, we formulate a class of parabolic translation-invariant PDEs with time-dependent inputs and, upon spatial discretization, express them as a nonautonomous dynamical system. Section~\ref{sec:review-dmd} contains a brief description of the Koopman operator theory and the DMD framework for construction of reduced-order representations of PDE-based models. In section~\ref{sec:online-lagrangian-dmd}, we present a local Lagrangian DMD, whose implementation shares relevant features of the time-varying DMD~\cite{zhang2017online} and the Lagrangian DMD~\cite{Lu_2020} to effectively represent translation-invariant PDEs with time-dependent inputs. Upper bounds of both the prediction error of our method and the operator norm error are derived in section~\ref{sec:theory}, as functions of reduction of rank and number of collected snapshots. This theoretical analysis demonstrates that the local Lagrangian DMD is more accurate than either time-varying DMD or Lagrangian DMD alone. A series of numerical experiments, reported in section~\ref{sec:numerical-examples}, serve to demonstrate our approach and to verify the tightness of these error bounds. Main conclusions drawn from our study are summarized in section~\ref{sec:conclusions}, accompanied by a discussion of the method's limitations and future research.

\section{Problem Formulation} \label{sec:problem-formulation} We are concerned with the following class of partial differential equations (PDE) with variable coefficients for a quantity of interest $u(t,\mathbf{x})$, with $\mathbf{x}\in\Omega\subset \mathbb{R}^{d}$:
\[
    \frac{\partial u}{\partial t} + \nabla_{\mathbf{x}}\cdot(G(t, \mathbf{x}, u)u) = \nabla_{\mathbf{x}}\cdot(D(t,\mathbf{x}, u)\nabla_{\mathbf{x}}u), (t, \mathbf{x}) \in (0, t_f] \times \Omega 
\]
\begin{equation}\label{eqn:variable-advection-diffusion}
    u(t_0, \mathbf{x}) = u_0(\mathbf{x})
\end{equation}

We consider a semi-discrete method to simulate equation (\ref{eqn:variable-advection-diffusion}) by discretizing in the spatial variables $\mathbf{x}$. For simplicity, we assume the number of gird points is $n$ for each of the $d$ spatial dimensions. We arrive at a nonautonomous dynamical system of general form:
\[
    \frac{d\mathbf{u}}{dt} = \mathbf{N}(t, \mathbf{u})
\]
\begin{equation}\label{eqn:general-nonautonomous-system}
    \mathbf{u}(0) = \mathbf{u}_0
\end{equation} whose right-hand side describes the dynamics of the PDE in (\ref{eqn:variable-advection-diffusion}) with an explicit time-dependence. With respect to construction of ROMs, we will be primarily concerned with the discretized equations (\ref{eqn:general-nonautonomous-system}). Let $\mathbf{u}\in\mathcal{M}\subset \mathbb{R}^{n^d}$ denote the numerical solution, and $\mathbf{N}:\mathbb{R}^{+}\times \mathbb{R}^{n^d}\rightarrow \mathbb{R}^{n^d}$ is the discretized PDE operator. 

Let the temporal domain $[0, t_f]$ be discretized uniformly with step size $\Delta t$, and define $t_i = i\Delta t$, for $0 = t_0 < t_1 < \cdots < t_{m} = t_f$. Furthermore, define $\boldsymbol{\Phi}_{\Delta t}(\cdot ; t_i): \mathbb{R}^n\rightarrow\mathbb{R}^n$ as the discrete flow map associated with the system (\ref{eqn:general-nonautonomous-system}), and similarly the continuous flow map is denoted as $\Phi_t(\cdot; s)$, such that for any $t\le t'$:
\begin{equation}\label{eqn:general-nonautonomous-time-shift-continuous}
    \mathbf{u}(t') = \Phi_{t'}(\mathbf{u}(t); t) := \mathbf{u}(t) + \int_{t}^{t'}\mathbf{N}(s, \mathbf{u}(s))ds
\end{equation} Furthermore,
\begin{equation}\label{eqn:general-nonautonomous-time-shift}
    \mathbf{u}_{i+1} = \boldsymbol{\Phi}_{\Delta t}(\mathbf{u}_i; t_i) := \mathbf{u}(t_i) + \int_{t_i}^{t_{i+1}}\mathbf{N}(s, \mathbf{u}(s))ds
\end{equation} where we define $\mathbf{u}_i = \mathbf{u}(t_i)$. 

\section{Review of DMD Algorithms} \label{sec:review-dmd} For the general dynamical system (\ref{eqn:general-nonautonomous-time-shift}), the associated family of Koopman operators evolve a set of observables along its flow. More precisely, given an observable function $g: \mathbb{R}^n\rightarrow\mathbb{R}^N$, the Koopman operator $\mathcal{K}_t^{t'}$ is defined such that:
\begin{equation}
    \mathcal{K}_t^{t'}g(\mathbf{u}(t)) := g(\mathbf{u}(t'))
\end{equation} For the discrete-time description (\ref{eqn:general-nonautonomous-time-shift}), we similarly define the associated Koopman operator $\mathcal{K}^{\Delta t}_i$, such that:
\begin{equation}\label{eqn:deiscrete-koopman-definition}
    \mathcal{K}_i^{\Delta t}g(\mathbf{u}_i) = g(\mathbf{u}_{i+1})
\end{equation} Both $\mathcal{K}_t^{t'}, \mathcal{K}_i^{\Delta t}$ are infinite-dimensional operators on the Hilbert space of all observable functions $g$. In addition, they are linear maps despite potential nonlinearity of the original system.

Dynamic mode decomposition (DMD) is a celebrated algorithm that attempts to approximate the eigenmodes of the Koopman operator to identify dominant frequencies and reconstruct the underlying dynamics from discrete observations. Let a training dataset containing $m$ collected snapshots be denoted as $\mathcal{S} = \{(\mathbf{g}_i,\mathbf{g}_{i+1}\})\}_{i=1}^m$, with $\mathbf{g}_i = g(\mathbf{u}_i)$. In line with (\ref{eqn:deiscrete-koopman-definition}), we would like to construct a best-fit linear operator $\mathbf{K}$ such that:
\begin{equation}
    \mathbf{g}_{i+1} \approx \mathbf{K}\mathbf{g}_i
\end{equation} for all $i=1,2,\ldots, m$.

\subsection{Standard DMD} The standard DMD algorithm attempts to reconstruct directly in solution space, i.e. $g(\mathbf{u}_i) = \mathbf{u}_i$ and $\mathbf{K}$ is constructed via minimizing the mean squared error (MSE):
\begin{equation}\label{eqn:standard-mse-loss-function}
    L_{\mathcal{S}}(\mathbf{K}) = \frac1m\sum_{i=1}^m\lvert\lvert \mathbf{y}_i - \mathbf{K}\mathbf{x}_i\rvert\rvert_{2}^2
\end{equation} where the pairs $(\mathbf{x}_i, \mathbf{y}_i) = (\mathbf{u}_{i}, \mathbf{u}_{i+1})$ form the data matrices of size $n \times m$:
\begin{equation} \label{eqn:data-matrices}
    \mathbf{X} = 
    \begin{bmatrix}
        \vline & \vline & \cdots & \vline \\
        \mathbf{u}_1 & \mathbf{u}_2 & \cdots & \mathbf{u}_m \\
        \vline & \vline & \cdots & \vline
    \end{bmatrix}, 
    \mathbf{Y} = 
    \begin{bmatrix}
        \vline & \vline & \cdots & \vline \\
        \mathbf{u}_2 & \mathbf{u}_3 & \cdots & \mathbf{u}_{m+1} \\
        \vline & \vline & \cdots & \vline
    \end{bmatrix} 
\end{equation}

The minimizer of (\ref{eqn:standard-mse-loss-function}) can be explicitly derived as:
\begin{equation}\label{eqn:standard-dmd-solution}
    \mathbf{K} = \mathbf{Y}\mathbf{X}^{\dagger}
\end{equation} where $\dagger$ denotes the Moore-Penrose pseudoinverse, $\mathbf{X}^{\dagger} = (\mathbf{X}^T\mathbf{X})^{-1}\mathbf{X}^*$. In order to compute $\mathbf{X}^{\dagger}$ stably and tractably, a truncated singular value decomposition (SVD) is often applied on the data matrix $\mathbf{X}$:
\begin{equation}
    \mathbf{X} \approx \mathbf{U}_r\boldsymbol{\Sigma}_r\mathbf{V}_r^*
\end{equation} where the subscript $r$ denotes a pre-specfied rank typically determined on a Frobenius-norm error threshold, with $\mathbf{U}_r\in\mathbb{R}^{n\times r}, \mathbf{V}_r\in \mathbb{R}^{m\times r}, \boldsymbol{\Sigma}_r\in\mathbb{R}^{r\times r}$ is a diagonal matrix containing the singular values $\sigma_1\ge \sigma_2\ge \cdots\ge \sigma_r$ of $\mathbf{X}$, in non-increasing order. Furthermore, the columns of $\mathbf{U}_r$ span a $r$ dimensional subspace of $\mathbb{R}^n$, making it a computationally efficient strategy to first project the observations, compute predictions on the lower-dimensional space, and transform back to the original state space \cite{doi:10.1137/1.9781611974508.ch1}. The procedure is summarized in Algorithm~\ref{alg:standard-dmd-algorithm}, which provides a continuous and fully data-driven model satisfying (\ref{eqn:standard-mse-loss-function}). The standard DMD algorithm serves as the foundation of a wide range of DMD algorithms that incorporate additional control parameters \cite{proctor2014dynamic, LU2021110550}.

\begin{algorithm}
\caption{Standard DMD Algorithm \cite{doi:10.1137/1.9781611974508.ch3}}
\label{alg:standard-dmd-algorithm}
\begin{algorithmic}
\STATE{Inputs: Data matrices $\mathbf{X},\mathbf{Y}$, containing $m$ snapshots of system states of dimension $n$. SVD truncation error threshold $\epsilon>0$:
\begin{equation}
    \mathbf{X} = 
    \begin{bmatrix}
        \vline & \vline & \cdots & \vline \\
        \mathbf{u}_1 & \mathbf{u}_2 & \cdots & \mathbf{u}_m \\
        \vline & \vline & \cdots & \vline
    \end{bmatrix}, 
    \mathbf{Y} = 
    \begin{bmatrix}
        
        \vline & \vline & \cdots & \vline \\
        \mathbf{u}_2 & \mathbf{u}_3 & \cdots & \mathbf{u}_{m+1} \\
        \vline & \vline & \cdots & \vline
    \end{bmatrix}
\end{equation}
} 

1. Compute SVD of $\mathbf{X}$ with truncation rank $r$:
\begin{equation}
    \mathbf{X} \approx \mathbf{U}_r\boldsymbol{\Sigma}_r\mathbf{V}^*_r
\end{equation} such that:
\begin{equation}
    \norm{\mathbf{X} - \mathbf{U}_r\boldsymbol{\Sigma}_r\mathbf{V}_r^*}_F^2 = 
    \frac{\sum_{k=m+1}^n\sigma_k^2}{\sum_{k=1}^{n}\sigma_k^2} < \epsilon
\end{equation}

2. Compute low-rank approximation to Koopman operator:
\begin{equation}\label{eqn:low-rank-koopman}
    \widehat{\mathbf{K}} = \mathbf{U}_r^*\mathbf{Y}\mathbf{V}_r\boldsymbol{\Sigma}_r^{-1}
\end{equation}

3. Compute eigendecomposition of $\widehat{\mathbf{K}}$:
\begin{equation}
    \widehat{\mathbf{K}} = \mathbf{Q}\boldsymbol{\Lambda}\mathbf{Q}^{-1}
\end{equation} where $\boldsymbol{\Lambda}$ is an $r\times r$ diagonal matrix containing the eigenvalues of $\widehat{\mathbf{K}}$.

4. Obtain the projected DMD modes:
\begin{equation}
    \boldsymbol{\Phi} = \mathbf{U}_r\mathbf{Q}
\end{equation} associated with eigenvalues $\lambda_k = \boldsymbol{\Lambda}_{kk}$.

5. Compute DMD predictions:
\begin{equation}
    \mathbf{u}_{\text{DMD}}(t) = \boldsymbol{\Phi}e^{t\boldsymbol{\Omega}}\boldsymbol{\Phi}^{\dagger}\mathbf{x}_0
\end{equation} where $\boldsymbol{\Omega}$ is a diagonal matrix of DMD frequencies with entries $\omega_{k} = \ln \lambda_{k}/{\Delta t}$. In particular, $\mathbf{u}_{\text{DMD}}(t_k) = \boldsymbol{\Phi}\boldsymbol{\Lambda}^k\boldsymbol{\Phi}^{\dagger}\mathbf{x}_0$ at temporal grid point $k$.
\end{algorithmic} 
\end{algorithm}

\subsection{Physics-Aware DMD} \label{sec:lagrangian-dmd-section} To account for fundamental physical constraints for problems describing conservative advection (i.e. non-negativity of solutions, mass conservation), reduced-order models in a Lagrangian frame of reference are first discussed in \cite{mojgani2017lagrangian} based on principal orthogonal decomposition (POD). In the data-driven Koopman operator formulation, the physics-aware DMD (or Lagrangian DMD) was developed in \cite{Lu_2020} for advection-dominated phenomena, where standard DMD fails. The main idea is to include the moving Lagrangian grid as observables in addition to a high-fidelity numerical solution. More explicitly, we consider the PDE (\ref{eqn:variable-advection-diffusion}) along the characteristic lines:
\begin{equation}\label{eqn:lagrangian-pde}
    \begin{dcases}
        \frac{d\mathcal{X}(t)}{dt} = G(t, \mathcal{X}(t), \tilde{u}(t, \mathcal{X}(t))) \\
        \frac{du(t, x)}{dt}\bigg|_{x = \mathcal{X}(t)} = \big[
            \nabla_{\mathbf{x}}\big(
                D(t, x, {u}(t,x)
        \big)\nabla_{\mathbf{x}}{u}(t,x)
        \big]\bigg|_{x=\mathcal{X}(t)}
    \end{dcases}
\end{equation} with initial conditions:
\begin{equation}
    \begin{dcases}
        \mathcal{X}_i(0) = \mathbf{x}_i \in \mathbb{R}^d\\
        {u}(0, \mathcal{X}_i(0)) = u_0(\mathcal{X}_i)
    \end{dcases}
\end{equation} where $\mathcal{X}_i$ denotes the $i$th point in the Lagrangian moving grid at which the solution to (\ref{eqn:variable-advection-diffusion}) is evaluated, denoted as $\tilde{u}(t, \mathcal{X}(t))$. The starting grid is assumed to be the same spatial discretization as that of (\ref{eqn:general-nonautonomous-system}). In particular, $\tilde{u}(t,\mathcal{X}(t))$ differs from the solution $u(t,x)$ of (\ref{eqn:general-nonautonomous-system}), which is in the Eulerian frame of reference. The solution on the Lagrangian grid can be interpolated to the Eularian grid, and vice versa \cite{Lu_2020}.

After discretizing (\ref{eqn:lagrangian-pde}), the Lagrangian system (\ref{eqn:lagrangian-pde}) yields a dynamical system of general form (\ref{eqn:general-nonautonomous-system}) with state variables:
\begin{equation}\label{eqn:l-dmd-states}
    \mathbf{w}(t) = 
    \begin{bmatrix}
        \boldsymbol{\mathcal{X}}(t) \\
        {\mathbf{u}}(t)
    \end{bmatrix} \in \mathbb{R}^N
\end{equation} where the effective state dimension $N = dn+n^d$, including the discretized solution $u(\mathbf{x}_i)$ at each spatial grid points and re-ordered into a vector, along with a one-dimensional grid for each of the $d$ spatial dimensions. The physics-aware DMD then considers the observables defined by $g(\mathbf{u}_i) = \mathbf{w}_i$, and the associated data matrices are:
\begin{equation} \label{eqn:ldmd-data-matrices}
    \mathbf{X} = 
    \begin{bmatrix}
        \vline & \vline & \cdots & \vline \\
        \mathbf{w}_1 & \mathbf{w}_2 & \cdots & \mathbf{w}_m \\
        \vline & \vline & \cdots & \vline
    \end{bmatrix}, 
    \mathbf{Y} = 
    \begin{bmatrix}
        \vline & \vline & \cdots & \vline \\
        \mathbf{w}_2 & \mathbf{w}_3 & \cdots & \mathbf{w}_{m+1} \\
        \vline & \vline & \cdots & \vline
    \end{bmatrix}
\end{equation}

\begin{remark}
    The formulation of state vector $\mathbf{w}(t)$ in (\ref{eqn:l-dmd-states}) suffers from the so-called curse of dimensionality as the PDE solution is defined on a $d$-dimensional spatial grid. Furthermore, the interpolation from $\tilde{u}(t,\mathbf{x})$ to $u(t,\mathcal{X}(t))$ requires the formation of meshgrids at each time step $t$. As observed in \cite{Lu_2020}, the Lagrangian DMD for advection-dominated phenomena is restricted to the use of low-dimensional problems. Although possible model order reduction techniques exist, such as using tensor-network based methods \cite{richter2021solving,chen2023combining}, the discussion of high-dimensional PDE solutions is out of the scope of this paper.
\end{remark}

\subsection{Time-Varying DMD} \label{sec:define-online-dmd} The time-varying DMD algorithm divides the temporal domain $[0,t_f]$ into $p$ sub-intervals, $[t_0, t_1], \ldots, [t_{p-1}, t_p]$, with $t_0=0, t_p=t_f$. For simplicity, we assume each sub-interval contains $r$ snapshots and $m = pr$. The time-varying DMD model introduces a time dependence to the linear operator, such that:
\begin{equation}
    \mathbf{g}_{i+1} \approx \mathbf{K}(t_{i})\mathbf{g}_i
\end{equation} which approximates the nonautonomous Koopman operator (\ref{eqn:deiscrete-koopman-definition}). A common construction of $\mathbf{K}(t)$ is piecewise constant in time, considered in this work, via solving $p$ minimization problems:
\begin{equation}\label{eqn:time-varying-dmd-minimization}
    \min_{\mathbf{K}_1,\ldots,\mathbf{K}_p}L_{\mathcal{S}}(\mathbf{K}(t)) = \min_{\mathbf{K}_1,\ldots,\mathbf{K}_p}\sum_{i=1}^pL_{\mathcal{S}_i}(\mathbf{K}_i)
\end{equation} with $\mathcal{S}_i$ being the snapshots collected from $[t_{i-1},t_i]$, and $\mathcal{S} = \bigcup_{i=1}^p\mathcal{S}_i$. The linear operator $\mathbf{K}^{(i)}$ can be interpreted as a local best-fit given by a standard DMD procedure on interval $[t_{i-1}, t_i]$.

\begin{equation}\label{eqn:time-varying-dmd-solution}
    \mathbf{K}(t) = \sum_{i=1}^p\mathbf{K}^{(i)}\delta_{[t_{i-1},t_i]}(t)
\end{equation} where $\delta_{[t_{i-1},t_i]}$ is the indicator function for time interval $[t_{i-1},t_i]$. It is also possible to construct other parameterized models of $\mathbf{K}(t)$, such as basis polynomials or a universal function approximator \cite{doi:10.1137/20M1342859}.

\section{Proposed Methodology} \label{sec:online-lagrangian-dmd} Both the standard DMD model and the physics-aware DMD model assume the underlying dynamical system (\ref{eqn:general-nonautonomous-system}) is autonomous or periodic, such that the Koopman operator (\ref{eqn:deiscrete-koopman-definition}) may be captured on a time-invariant manifold given sufficient observations. Furthermore, the standard DMD algorithm tends to perform poorly on phenomena with advective mass and sharp gradients due to oscillatory DMD modes \cite{Lu_2020}. Although the physics-aware DMD is sufficient for prediction of spatially-dependent advection phenomena, the inherent assumption of time homogeneity gives rise to model misspecification and degradation of accuracy for time-dependent advection problems (\ref{eqn:variable-advection-diffusion}). To address the inaccuracies introduced by both standard DMD and physics-aware DMD, we consider the following procedure, summarized in Algorithm~\ref{alg:time-dependent-lagrangian-dmd}, which effectively introduces a time-dependence to the Lagrangian reduced order model.

\begin{algorithm}
\caption{Time-varying physics-aware DMD}
\label{alg:time-dependent-lagrangian-dmd}
\begin{algorithmic}
\STATE{Inputs: Data matrices $\mathbf{X},\mathbf{Y}$, containing $m$ snapshots of enlarged states of dimension $N$. SVD truncation error threshold $\epsilon>0$:
\begin{equation}
    \mathbf{X} = 
    \begin{bmatrix}
        \vline & \vline & \cdots & \vline \\
        \mathbf{g}_1 & \mathbf{g}_2 & \cdots & \mathbf{g}_m \\
        \vline & \vline & \cdots & \vline
    \end{bmatrix}, 
    \mathbf{Y} = 
    \begin{bmatrix}
        
        \vline & \vline & \cdots & \vline \\
        \mathbf{g}_2 & \mathbf{g}_3 & \cdots & \mathbf{g}_{m+1} \\
        \vline & \vline & \cdots & \vline
    \end{bmatrix}
\end{equation}
} where $\mathbf{g}_i = g(\mathbf{u}_i) = \mathbf{w}_i$ are the concatenation of PDE solution and Lagrangian grid, as defined in equation (\ref{eqn:l-dmd-states}).

1. For $k = 1, 2, \ldots, p$, compute low-rank approximation to Koopman operator for interval $[t_{i-1}, t_i]$ by applying the standard DMD algorithm in Algorithm~\ref{alg:standard-dmd-algorithm} to the data matrices:
\[
    \mathbf{X}_i = 
    \begin{bmatrix}
        \vline & \vline & \cdots & \vline \\
        \mathbf{g}_{(i-1)r+1} & \mathbf{g}_{(i-1)r+2} & \cdots & \mathbf{g}_{ir} \\
        \vline & \vline & \cdots & \vline
    \end{bmatrix}
\]
\[
    \mathbf{Y}_i =
    \begin{bmatrix}
        
        \vline & \vline & \cdots & \vline \\
        \mathbf{g}_{(i-1)r+2} & \mathbf{g}_{(i-1)r+3} & \cdots & \mathbf{g}_{ir+1} \\
        \vline & \vline & \cdots & \vline
    \end{bmatrix}
\] 

2. Obtain DMD modes:
\begin{equation}
    \boldsymbol{\Phi}^{(i)} = \mathbf{U}_r^{(i)}\mathbf{Q}^{(i)}
\end{equation} where $\mathbf{U}^{(i)}$ contains the left singular vectors of data matrix $\mathbf{X}_i$, and $\mathbf{Q}^{(i)}$ contains the eigenvetors of operator $\widehat{\mathbf{K}}^{(i)}$, analogous to (\ref{eqn:low-rank-koopman}) of Algorithm~\ref{alg:standard-dmd-algorithm}, along with eigenvalues: $\boldsymbol{\Lambda}^{(1)}, \boldsymbol{\Lambda}^{(2)}, \ldots, \boldsymbol{\Lambda}^{(p)}$.

3. Obtain prediction for $t\in [t_{i}, t_{i+1}]$:
\begin{equation}
   \mathbf{w}_{\text{DMD}}(t_i) = 
     \bigg(\prod_{k=1}^{i-1}\boldsymbol{\Phi}^{(k)}e^{(t_k-t_{k-1})\boldsymbol{\Omega}^{(k)}}[\boldsymbol{\Phi}^{(k)}]^{\dagger}\bigg)\mathbf{w}_0
\end{equation} then continue to evolve for an additional time $(t-t_{i})$ and obtain the final prediction:
\begin{equation}
    \mathbf{w}_{\text{DMD}}(t) = \boldsymbol{\Phi}^{(i)}e^{(t-t_{i-1})\boldsymbol{\Omega}^{(i)}}[\boldsymbol{\Phi}^{(i)}]^{\dagger}\mathbf{w}_{\text{DMD}}(t_{i})
\end{equation}

4. Transform back to Eulerian solution:
\begin{equation}
    \mathbf{u}_{\text{DMD}}(t) = g^{-1}(\mathbf{w}_{\text{DMD}}(t))
\end{equation}
\end{algorithmic} 
\end{algorithm}

\begin{remark}
    Algorithm~{\ref{alg:time-dependent-lagrangian-dmd}} provides an elementary implementation of the (temporal) piece-wise constant Koopman operator in (\ref{eqn:time-varying-dmd-solution}). Upon appropriate modifications of $\mathbf{K}(t)$ to allow superpositions of DMD frequencies in each time interval, it is possible to recover other forms of DMD strategies, such as the multi-resolution DMD of \cite{kutz2015multiresolution} or the windowed DMD of \cite{alfatlawi2020incremental}. 
\end{remark}

\begin{remark} In terms of computational complexity, it is possible to consider the incremental SVD updates with adaptive rank truncation to directly update $\mathbf{K}^{(i)}$ to $\mathbf{K}^{(i+1)}$ in low-rank format \cite{10.1007/3-540-47969-4_47}. However, due to the inclusion of Lagrangian moving grids in the formulation of (\ref{eqn:l-dmd-states}), it is assumed that the data matrices have dimensions $N\gg m$ and are of full column rank. The size constraint is especially true in high-dimensional PDE problems. In our numerical experiments, we did not observe a significant computational advantage of applying incremental SVD update to computed operators $\mathbf{K}^{(1)}, \ldots, \mathbf{K}^{(i)}$. In particular, a direct pseudoinverse computation in standard DMD involves $O(m^2N)$ runtime complexity, which is asymptotically more expensive than $p$ separate SVD computations, yielding $O(pr^2N) = O(mrN)$, with $m = pr$. A small computational saving may be achievable if the highest rank of data matrices during each time interval of collected snapshots is bounded by some $r' < r$, in which case the runtime complexity is $O(p\cdot rr'N) = O(mr'N)$, by applying incremental SVD updates. 
\end{remark}

\section{Theoretical Analysis} \label{sec:theory} The judicious choice of subintervals in the time-varying DMD formulation \ref{sec:define-online-dmd} is crucial for prediction accuracy. As general guideline, we first present in Section~\ref{subsec:prediction-error-bounds} pointwise and average error upper bounds for the time-varying DMD in (\ref{eqn:time-varying-dmd-solution}). In Section~\ref{subsec:perturbation-bounds}, we compute upper bounds of perturbations to the learned operator in terms of $L^2$ operator norm under truncation of frequencies and deletion of training data. Furthermore, for classes of linear dynamical systems, the bounds can be refined by analyzing the norm of time-shifted training data $\mathbf{Y}$ in relation to that of $\mathbf{X}$. For general nonlinear dynamical systems, we refer the reader to the analysis provided in the analysis given in Section 3 of \cite{doi:10.1137/20M1342859}

\subsection{Prediction Error} \label{subsec:prediction-error-bounds}

We first consider the pointwise prediction error of time-varying DMD strategy:
\begin{proposition}[Pointwise error for time-varying DMD]\label{thm:cumulative-prediction-error} Assume the system in equation (\ref{eqn:general-nonautonomous-system}) and the time-varying DMD in (\ref{eqn:time-varying-dmd-minimization}) satisfy the following properties:

\begin{enumerate}
    \item $\mathbf{N}(t,\cdot): \mathbb{R}^{n}\rightarrow\mathbb{R}^n$ is uniformly Lipschitz (in time) with constant $L>0$.
    \item $\sup_{s\in [t_0, t_f]}\norm{\boldsymbol{\Phi}_{\Delta t}(\cdot; s) - \mathbf{K}(s)}_{L^{\infty}(\mathbb{R}^n)} < +\infty$, where $\mathbf{K}(s)$ is piecewise constant on each interval $s\in [t_0,t_1],[t_1, t_2], \ldots, [t_{p-1}, t_p]$.  $\mathbf{K}_1,\mathbf{K}_2,\ldots, \mathbf{K}_p$ are respective solutions of the standard problem of minimizing (\ref{eqn:standard-mse-loss-function}) on each interval $[t_0, t_1], [t_1, t_2], \ldots, [t_{p-1}, t_p]$.
    
    \item All reconstructed solutions $\mathbf{x}_{\text{DMD}}$ belong to the solution manifold, defined as:
    \begin{equation}
        \mathcal{M}_{\Delta t} = \{\mathbf{x}\in \mathcal{M}: \pphi_{\Delta t}(\mathbf{x}; t_i) \in \mathcal{M}\}
    \end{equation}
\end{enumerate}

Define the error of incremental DMD at time step $t_n$ to be:
\begin{equation}
    \mathcal{E}^n = \norm{\mathbf{x}_n - \widehat{\mathbf{x}}_n}_2^2
\end{equation} where $\mathbf{x}_n = \mathbf{x}(t_n)$ is exact, and $\widehat{\mathbf{x}}_n$ is the approximation given by DMD. Rewritting the model expression:
\begin{equation}
    \widehat{\mathbf{x}}_{k+1} = \mathbf{K}(t_k)\widehat{\mathbf{x}}_k = \widehat{\mathbf{x}}_k + \mathbf{A}(t_k)\widehat{\mathbf{x}}_k
\end{equation} where:
\begin{equation}
    \mathbf{A}(t) := \mathbf{K}(t) - \mathbf{I}_N
\end{equation} then the pointwise error of time-vary DMD is:
\begin{equation}
    \mathcal{E}^{n} \le (1+e^{L\Delta t})^m\mathcal{E}_0 + \sum_{j=1}^p\sum_{l=0}^r(1+e^{L\Delta t})^l \norm{\pphi_{\Delta t}(\cdot; t_{n-(jr-l)}) - \mathbf{A}_k}^2_{L^{\infty}(\mathcal{M}_{\Delta t})}
\end{equation} 
\end{proposition}
\begin{proof}
     For ease of presentation, we omit the time dependence in the flow map and let $\boldsymbol{\Phi}_{\Delta t}(\mathbf{x}(t)) = \boldsymbol{\Phi}_{\Delta t}(\mathbf{x}(t); t)$, and $\norm{\cdot}_{{\infty}} = \norm{\cdot}_{L^{\infty}(\mathcal{M}_{\Delta t})}$. By Gronwall's inequality along with Lipschitz continuity, we have for any time $t$ and solutions $\mathbf{x}, \widehat{\mathbf{x}} \in \mathcal{M}_{\Delta t}$:
    \begin{equation}
        \norm{\pphi_{\Delta t}(\mathbf{x}(t)) - \pphi_{\Delta t}(\widehat{\mathbf{x}}(t))}_2 \le e^{\tau L}\norm{\mathbf{x}(t) - \widehat{\mathbf{x}}(t)}_2, \tau \in [0,\Delta t]
    \end{equation}
Then by repeated applications of triangle inequality:
\[
    \mathcal{E}^n = 
        \norm{\mathbf{x}_{n-1} + \pphi_{\Delta t}(\mathbf{x}_{n-1}) - (\widehat{\mathbf{x}}_{n-1} + \mathbf{A}(t_{n-1})\widehat{\mathbf{x}}_{n-1})}_2^2
        \\
        \le 
        \norm{\mathbf{x}_{n-1}-\widehat{\mathbf{x}}_{n-1}}_2^2 + 
        \norm{\pphi_{\Delta t}(\mathbf{x}_{n-1}) - \mathbf{A}(t_{n-1})\widehat{\mathbf{x}}_{n-1}}_2^2 \\
        = \norm{\mathbf{x}_{n-1}-\widehat{\mathbf{x}}_{n-1}}_2^2 +
        \norm{\pphi_{\Delta t}(\mathbf{x}_{n-1}) - \pphi_{\Delta t}(\widehat{\mathbf{x}}_{n-1}) + \pphi_{\Delta t}(\widehat{\mathbf{x}}_{n-1}) - \mathbf{A}(t_{n-1})\widehat{\mathbf{x}}_{n-1}}_2^2 \\
        \le \norm{\mathbf{x}_{n-1}-\widehat{\mathbf{x}}_{n-1}}_2^2
        +
        \norm{\pphi_{\Delta t}(\mathbf{x}_{n-1}) - \pphi_{\Delta t}(\widehat{\mathbf{x}}_{n-1})}_2^2 +
        \norm{\pphi_{\Delta t}(\widehat{\mathbf{x}}_{n-1}) - \mathbf{A}(t_{n-1})\widehat{\mathbf{x}}_{n-1}}_2^2 \\
        \le \mathcal{E}^{n-1} + e^{\Delta t L}\mathcal{E}^{n-1} + \norm{\pphi_{\Delta t}(\cdot; t_n) - \mathbf{A}_p}_{{\infty}}^2 \\
        \le (1+e^{\Delta t L})\mathcal{E}^{n-2} + (1+e^{\Delta t L})\norm{\pphi_{\Delta t}(\cdot; t_{n-1}) - \mathbf{A}_p}_{{\infty}}^2 + \norm{\pphi_{\Delta t}(\cdot; t_n) - \mathbf{A}_p}_{{\infty}}^2 \\
        \le \cdots \le (1+e^{\Delta t L})^r\mathcal{E}^{n-r} + \sum_{l=0}^r(1+e^{\Delta t L})^l\norm{\pphi_{\Delta t}(\cdot; t_{n-(r-l)}) - \mathbf{A}_m}_{{\infty}}^2 \\
        \le \cdots \le  
        (1+e^{\Delta t L})^{2r}\mathcal{E}^{n-2r} + \sum_{l=0}^r(1+e^{\Delta tL})^l\norm{\pphi_{\Delta t}(\cdot; t_{n-(r-l)}) - \mathbf{A}_m}_{\infty}^2 + \cdots \\
        \sum_{l=0}^r(1+e^{\Delta tL})^l\norm{\pphi_{\Delta t}(\cdot; t_{n-(2r-l)}) - \mathbf{A}_{m-1}}_{{\infty}}^2 \\
        \le \cdots \le 
        (1+e^{\Delta tL})^m\mathcal{E}_0 + \sum_{j=0}^p\sum_{l=0}^w(1+e^{L\Delta t})^l \norm{\Phi_{\Delta t}(\cdot; t_{n-(jr-l)}) - \mathbf{A}_k}^2_{{\infty}}
\]
\end{proof}

\begin{remark}
    If $\mathbf{K}(t) \equiv \mathbf{K}$ is constant in time, we recover the upper bound investigated in Theorem 4.3 of \cite{Qin_2019} and subsequently that in equation (3.11) of \cite{Lu_2021}.
\end{remark}

\begin{corollary} The time-varying DMD of (\ref{eqn:time-varying-dmd-solution}) is at least as accurate in the MSE sense as the standard DMD of (\ref{eqn:standard-dmd-solution}).
\end{corollary}
\begin{proof}
    The property can be intuitively interpreted from the fact that a stepwise constant (in time) approximation is always at least as good on average as a constant approximation. More precisely, let $\mathbf{K},\mathbf{K}(t)$ denote the solutions of standard DMD and time-varying DMD, respectively, we may rewrite the minimization problem in (\ref{eqn:standard-mse-loss-function}):
\[
        L_{\mathcal{S}}(\mathbf{K}) = \min_{\mathbf{K}}\frac{1}{m}\sum_{i=1}^m\norm{\mathbf{y}_i - \mathbf{K}\mathbf{x}_i}_2^2 = \min_{\mathbf{K}}\frac{1}{p}\sum_{i=1}^{p}\frac{1}{w}\sum_{j=1}^{w}\norm{\mathbf{y}_{n-(ir-j)} - \mathbf{K}\mathbf{x}_{n-(ir-j)}}_2^2 
\] and by definition of minimum, we conclude:
\[
        L_{\mathcal{S}}(\mathbf{K}) \ge \frac1p\sum_{i=1}^p\min_{\mathbf{K}_i}\frac1w\sum_{j=1}^w\norm{\mathbf{y}_{n-(iw-j)} - \mathbf{K}_i\mathbf{x}_{n-(iw-j)}}_2^2 = L_{\mathcal{S}}(\mathbf{K}(t))
\]
\end{proof}

\subsection{Perturbation Analysis}\label{subsec:perturbation-bounds} With the DMD algorithms introduced in Section~\ref{sec:review-dmd}, we provide an operator 2-norm error bound on the DMD solution for two cases of common operations in engineering: (1) truncation of singular value decomposition (SVD) rank in data matrix $\mathbf{X}$ and, (2) deletion of most recent snapshots in both $\mathbf{X}, \mathbf{Y}$. In particular, we connect the error bound with a case study of nonautonomous linear dynamical system with the following form:
\begin{equation}\label{eqn:time-dep-linear-system-with-force}
    \begin{dcases}
        \frac{d\mathbf{u}(t)}{dt} = \mathbf{C}(t)\mathbf{u}(t) + \mathbf{f}(t)\\
        \mathbf{u}(0) = \mathbf{u}_0
    \end{dcases}
\end{equation} whose solution is provided:
\begin{equation}
    \mathbf{u}(t) = \Phi_t(\mathbf{u}_0; 0) = \exp\bigg(
        \int_0^t\mathbf{C}(s)ds
    \bigg)\mathbf{u}_0 + 
    \int_0^t\exp\bigg(
        \int_{s}^t\mathbf{C}(\tau)d\tau
    \bigg)\mathbf{f}(s)ds
\end{equation} 

We first present the results without assumptions on the underlying system.

\begin{proposition}\label{eqn:op-norm-bound-truncating-rank} (Operator norm error under rank truncation) Let the SVD of data matrix $\mathbf{X} = \mathbf{U}\boldsymbol{\Sigma}\mathbf{V}^T$. $\boldsymbol{\Sigma}$ contains the singular values arranged in non-increasing order, i.e. $\sigma_1 = \sigma_{\text{max}} \ge \sigma_2 \ge\cdots\ge \sigma_{\text{min}} = \sigma_{\text{rank}(\mathbf{X})}$. Let a truncated SVD with $r\le \text{rank}(\mathbf{X})$ be denoted as $\mathbf{X}_r = \mathbf{U}_r\ssigma_r\mathbf{V}_r^T$ where only the first $r$ columns are retained in $\mathbf{U}_r,\mathbf{V}_r$, and the first $r$ singular values are retained in $\ssigma_r$. Then the operator norm error has the following upper bound:
\begin{equation}
    \norm{\mathbf{A}-\mathbf{A}_r}_2 \le \frac{\sigma_{\text{max}}(\mathbf{Y})}{\sigma_{\text{min}}(\mathbf{X})}
\end{equation}
\end{proposition}
\begin{proof}
    \begin{equation}
        \norm{\mathbf{K} - \mathbf{K}_r}_2^2 = \norm{\mathbf{Y}\mathbf{X}^{\dagger} - \mathbf{YX}_r^{\dagger}}_2^2 \le \norm{\mathbf{Y}}_2^2\cdot\norm{\mathbf{X}^{\dagger} - \mathbf{X}_r^{\dagger}}_2^2 = 
        \frac{\sigma_{max}^2(\mathbf{Y})}{\sigma_{min}^2(\mathbf{X})}
    \end{equation}
\end{proof} 

\begin{remark}
    The bound presented in Proposition~\ref{eqn:op-norm-bound-truncating-rank} is an upper bound in the sense that it does not depend on the rank-$r$ due to the pseudoinverse operation. More granular bounds can be derived by analyzing instead the pointwise error for a specific observation $\mathbf{x}$:
    \begin{equation}\label{eqn:pointwise-svd-bound}
        \norm{\mathbf{K}\mathbf{x} - \mathbf{K}_r\mathbf{x}}_2^2 \le \sigma^2_{{max}}(\mathbf{Y})\norm{
        \sum_{k=r}^{\text{rank}(\mathbf{X)}}
        -\frac{1}{\sigma_k(\mathbf{X})}(\mathbf{u}_k^T\mathbf{x})\mathbf{v}_k
        }_2^2 
    \end{equation}
$$ = \sum_{k=r}^{\text{rank}(\mathbf{X})}\frac{\sigma_{{max}}^2(\mathbf{Y})}{\sigma_k^2(\mathbf{X})}(\mathbf{u}_k^T\mathbf{x})^2
$$
     Under different assumptions of $\mathbf{x}$ in relations to the column space of data matrix $\mathbf{X}$, the bound (\ref{eqn:pointwise-svd-bound}) can be tightened \cite{396bf6e1-ef54-3bf6-a49b-862db8404076}. 
\end{remark}

To analyze the time-varying DMD strategy in Section~\ref{sec:define-online-dmd}, one may view the individual solutions $\mathbf{K}_i$ on time interval $[t_{i-1}, t_i]$ as a standard DMD solution with fewer observations. To provide a benchmark on the effect of adding/deleting observations in the training data and investigate dependencies, we illustrate the operator norm perturbation that occurs by deleting the most recent observation. The general case of deleting $r$ most recent observations can be analogously derived using the Sherman-Morrison-Woodbury update formula. For the pseudoinverse of data matrices, the following result holds: 

\begin{lemma}\label{eqn:updating-pseudoinverse} (Updating pseudoinverse) Suppose $N\ge m$ and $\mathbf{X}_m\in \mathbb{R}^{N\times m}$ has full column rank, Furthermore, let $\mathbf{u}\in\mathbb{R}^N$ be a newly collected snapshot, the pseudoinverse of $\mathbf{X} = [\mathbf{X}_m, \mathbf{u}]\in\mathbb{R}^{N\times (m+1)}$ is given by:
\[
    \mathbf{X}^{\dagger} = 
    \begin{bmatrix}
        \mathbf{X}_m^{\dagger} + c\mathbf{X}_m^{\dagger}\mathbf{uu}^T(\mathbf{X}_m\mathbf{X}_m^{\dagger})^T - c\mathbf{X}_m^{\dagger}\mathbf{uu}^T \\
        -c\mathbf{u}^T(\mathbf{X}_m\mathbf{X}_m^{\dagger})^T + c\mathbf{u}^T
    \end{bmatrix}
\]
\[
    = 
    \begin{bmatrix}
        \mathbf{X}_m^{\dagger} \\
        \mathbf{0}
    \end{bmatrix} - 
    c
    \begin{bmatrix}
        \mathbf{X}_m^{\dagger}\mathbf{u}\\
        1
    \end{bmatrix}((\mathbf{I} - \mathbf{X}_m\mathbf{X}_m^{\dagger})\mathbf{u})^T
\] where:
\begin{equation}
    c = 
    \frac{1}{\norm{\mathbf{u}}_2^2 - \mathbf{u}^T\mathbf{X}_m(\mathbf{X}_m^T\mathbf{X}_m)^{-1}\mathbf{X}_m^T\mathbf{u}} \ge 
    \frac{1}{\norm{\mathbf{u}}_2^2}
\end{equation} The lower bound is attained if $\mathbf{u}$ is orthogonal to the range of $\mathbf{X}_m$.
\end{lemma}
\begin{proof}
    We directly apply the block matrix inverse formula \cite{doi:10.1137/1.9781611971446} to $(\mathbf{X}^T\mathbf{X})^{-1}$:
 \[
         (\mathbf{X}^T\mathbf{X})^{-1} = 
       \begin{bmatrix}
        \mathbf{X}_m^T\mathbf{X}_m & \mathbf{X}_m^T\mathbf{u} \\
        \mathbf{u}^T\mathbf{X}_m & \norm{\mathbf{u}}_2^2
       \end{bmatrix}^{-1}
\]
\[
       = \begin{bmatrix}
           (\mathbf{X}_m^T\mathbf{X}_m)^{-1} + c\mathbf{X}_m^{\dagger}\mathbf{uu}^T(\mathbf{X}_m^{\dagger})^T & -c\mathbf{X}_m^{\dagger}\mathbf{u} \\
           -c\mathbf{u}^T(\mathbf{X}_m^{\dagger})^{T} & c
       \end{bmatrix}
\] and multiply the result to $\mathbf{X}^T = 
    \begin{bmatrix}
        \mathbf{X}_m^T \\
        \mathbf{u}^T
    \end{bmatrix}$.
\end{proof}

\begin{proposition}\label{eqn:op-norm-bound-truncating-column} (Operator 2-norm perturbation under column deletion) 

Let $\mathbf{X} = [\mathbf{X}_m, \mathbf{u}] \in \mathbb{R}^{N\times (m+1)}$, $\mathbf{Y} = [\mathbf{Y}_m, \mathbf{v}]\in \mathbb{R}^{N\times (m+1)}$, and $\mathbf{X}_m,\mathbf{Y}_m\in\mathbb{R}^{N\times m}$, with $N\ge m$. We further assume that $\mathbf{X}_m$ has full column rank. Then, the operator norm error satisfies the following upper bound:
\begin{equation} \label{eqn:column-deletion-upper-bound}
    \norm{\mathbf{K - \mathbf{K}}_m}_2 \le 
    \sqrt{
        c^2\norm{\mathbf{u}}_2^2\bigg(
        1 + \frac{\norm{\mathbf{u}}_2^2}{\sigma^2_{min}(\mathbf{X}_m)}
    \bigg)(\sigma_{max}^2(\mathbf{Y}_m) + \norm{\mathbf{v}}_2^2) +
    \frac{\norm{\mathbf{v}}_2^2}{\sigma_{min}^2(\mathbf{X}_m)}
    }
\end{equation} with $c$ defined in Lemma~\ref{eqn:updating-pseudoinverse}. In particular, if $\mathbf{u}$ is orthogonal to the range of $\mathbf{X}_m$, the bound is tightened to:
\begin{equation}
    \norm{\mathbf{K - \mathbf{K}}_m}_2 \le 
    \sqrt{
        \frac{\sigma_{max}^2(\mathbf{Y}_m) + \norm{\mathbf{v}}_2^2}{\norm{\mathbf{u}}_2^2} + 
        \frac{\sigma_{max}^2(\mathbf{Y}_m) + 2\norm{\mathbf{v}}_2^2}{\sigma_{\text{min}}^2(\mathbf{X}_m)}
    }
\end{equation}
\end{proposition}
\begin{proof}
\begin{eqnarray*}
    \norm{\mathbf{K} - \mathbf{K}_m}_2^2 = 
    \norm{\mathbf{YX}^{\dagger} - \mathbf{Y}_m\mathbf{X}_m^{\dagger}}_2^2 = 
    \norm{\mathbf{Y}\mathbf{X}^{\dagger} - \mathbf{Y}\widehat{\mathbf{X}_m}^{\dagger} + \mathbf{Y}\widehat{\mathbf{X}_m}^{\dagger} - \widehat{\mathbf{Y}_m}\widehat{\mathbf{X}_m}^{\dagger}}_2^2
\end{eqnarray*} where we define:
\begin{equation}
    \widehat{\mathbf{X}_m}^{\dagger} := 
    \begin{bmatrix}
        \mathbf{X}_m^{\dagger} \\
        \mathbf{0}_{1\times N}
    \end{bmatrix} \in \mathbb{R}^{(m+1)\times N}, 
    \widehat{\mathbf{Y}_m} = 
    \begin{bmatrix}
        \mathbf{Y}_m & \mathbf{0}_{N\times 1}
    \end{bmatrix}\in \mathbb{R}^{N\times (m+1)}
\end{equation} then by triangle inequality:
\begin{eqnarray*}
    \norm{\mathbf{K} - \mathbf{K}_m}_2^2 \le \norm{\mathbf{Y}}_2^2\norm{\mathbf{X}^{\dagger}-\widehat{\mathbf{X}_m}^{\dagger}}_2^2 + 
    \norm{\widehat{\mathbf{X}_m}^{\dagger}}_2^2\norm{\mathbf{Y} - \widehat{\mathbf{Y}_m}}_2^2
\end{eqnarray*} where $\norm{\mathbf{X}^{\dagger}-\widehat{\mathbf{X}_m}^{\dagger}}_2$ needs to be further bounded. Using Lemma~\ref{eqn:updating-pseudoinverse}, we have:
\begin{equation}
    \mathbf{X}^{\dagger}-\widehat{\mathbf{X}_m}^{\dagger} = 
    - 
    c
    \begin{bmatrix}
        \mathbf{X}_m^{\dagger}\mathbf{u}\\
        1
    \end{bmatrix}((\mathbf{I} - \mathbf{X}_m\mathbf{X}_m^{\dagger})\mathbf{u})^T
\end{equation} Furthermore, we have:
\begin{equation}
    \Bigg\lvert\Bigg\lvert
        \begin{bmatrix}
            \mathbf{X}_m^{\dagger}\mathbf{u} \\
            1
        \end{bmatrix}
    \Bigg\rvert\Bigg\rvert_2^2 \le 1 + \frac{\norm{\mathbf{u}}_2^2}{\sigma^2_{min}(\mathbf{X}_m)}
\end{equation} and as a projection matrix:
\begin{equation}
    \norm{\mathbf{I} - \mathbf{X}_m\mathbf{X}_m^{\dagger}}_2^2 \le 1
\end{equation} 

Then we may conclude:
\begin{equation}
    \norm{\mathbf{X}^{\dagger}-\widehat{\mathbf{X}_m}^{\dagger}}_2^2 \le c^2\norm{\mathbf{u}}_2^2\bigg(
        1 + \frac{\norm{\mathbf{u}}_2^2}{\sigma^2_{min}(\mathbf{X}_m)}
    \bigg)
\end{equation} 

Putting everything together, we conclude that:
\begin{equation}
    \norm{\mathbf{K - \mathbf{K}}_m}_2^2 \le 
    c^2\norm{\mathbf{u}}_2^2\bigg(
        1 + \frac{\norm{\mathbf{u}}_2^2}{\sigma^2_{min}(\mathbf{X}_m)}
    \bigg)(\sigma_{max}^2(\mathbf{Y}_m) + \norm{\mathbf{v}}_2^2) +
    \frac{\norm{\mathbf{v}}_2^2}{\sigma_{min}^2(\mathbf{X}_m)}
\end{equation} 

Under the assumption of $\mathbf{u}$ being orthogonal to $\text{range}(\mathbf{X}_m)$, the last conclusion follows by the reduction of lower bound for $c$ presented in Lemma~\ref{eqn:updating-pseudoinverse}.
\end{proof}

\begin{figure}
    \centering
    \includegraphics[width=8cm]{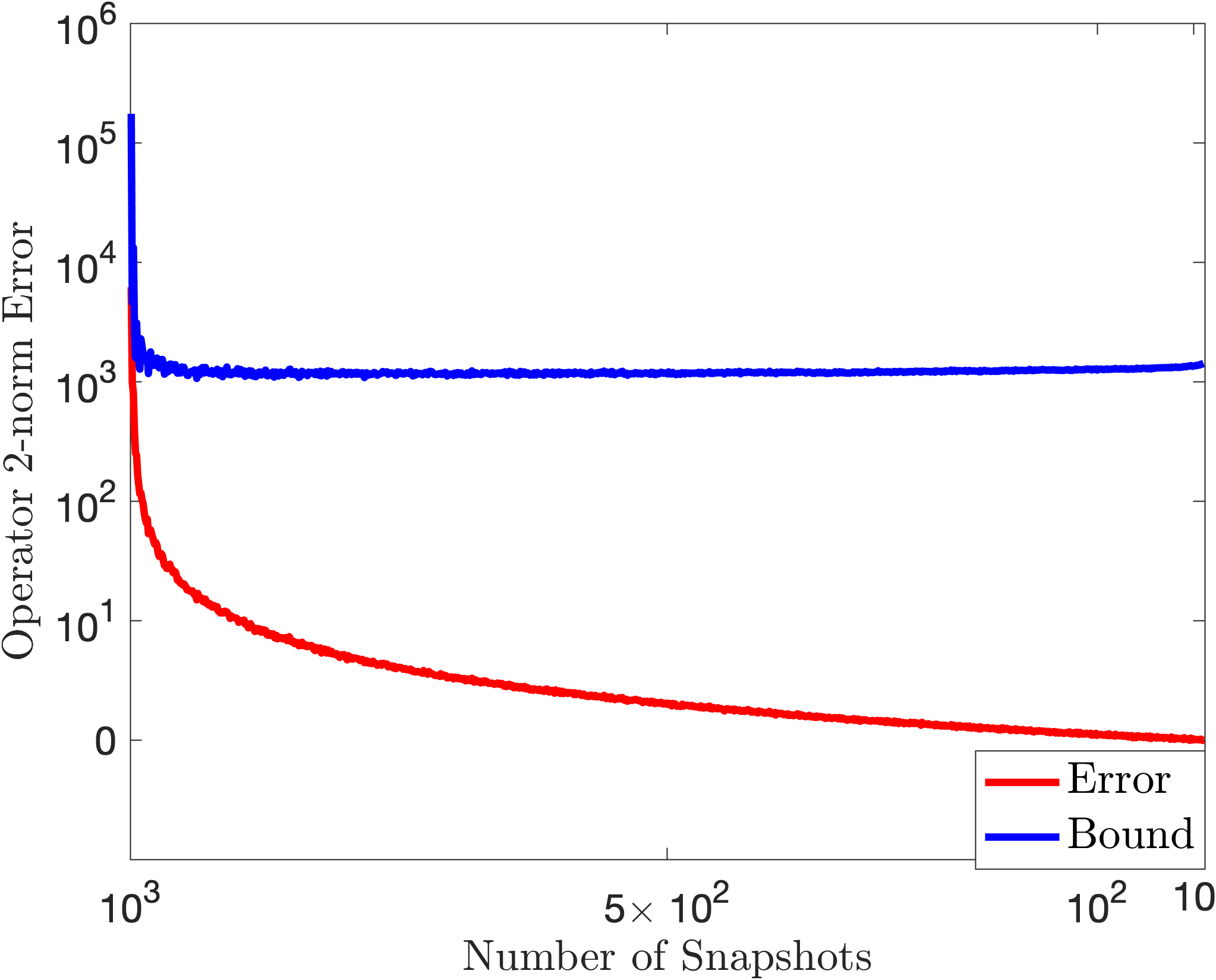}
    \caption{Operator norm error bound (\ref{eqn:op-norm-bound-truncating-column}) under deletion of most recent observation, for random Gaussian data matrices. The comparison of true operator norm error and upper bounds are averaged over 10 seeds. }
    \label{fig:delete-column-bound}
\end{figure}

Figure~\ref{fig:delete-column-bound} provides a verification of the bound in Theorem~\ref{eqn:op-norm-bound-truncating-column} using random Gaussian matrices, averaged over 10 random seeds. The results obtained in Theorem~\ref{eqn:op-norm-bound-truncating-rank} and Theorem~\ref{eqn:op-norm-bound-truncating-column} only rely on general linear algebra operations. 

With explicit form of the dynamical system, such as the system in equation (\ref{eqn:time-dep-linear-system-with-force}), more insights can be gained by leveraging the dependence of time-shifted data matrix $\mathbf{Y}$ on $\mathbf{X}$ via the flow map $\pphi_{\Delta t}$, as we now present in the following proposition:

\begin{proposition}\label{lemma:time-shift-norm-bounds} (Time-shift norm upper bound, for system (\ref{eqn:time-dep-linear-system-with-force})) Assume that $\mathbf{C}(t)$ is diagonalizable for all $t$, and $\mathbf{C}(t)$, $\mathbf{f}(t)$ are piecewise continuous on all intervals $[t_0,t_1],\ldots, [t_{m-1}, t_m]$. Then we have that the norm of $\mathbf{Y}$ is connected with the norm of $\mathbf{X}$ as the following, with $f,\gamma$ defined in equation (\ref{eqn:define-gamma}):
\[
    \norm{\mathbf{Y}}_2
    \le 
    \exp\bigg(\frac12\gamma^2\Delta t\bigg)\sqrt{\frac{mf^2}{\gamma^2} +
    \sum_{i=1}^{m}\sigma_i^2(\mathbf{X})}
\]
\end{proposition}
\begin{proof}
For convenience of notations, define the time-dependent matrix:
\begin{equation}
    \mathbf{M}_{\Delta t}^{(i)} = \mathbf{M}_{\Delta t}(t_i) := \exp\bigg(
        \int_{t_i}^{t_i+\Delta t}\mathbf{C}(s)ds
    \bigg)
\end{equation} and the time-dependent vector:
\begin{equation}
    \mathbf{g}_{\Delta t}^{(i)} = \mathbf{g}_{\Delta t}(t_i) = \int_{t_i}^{t_i+\Delta t}\exp\bigg(
        \int_s^{t_i+\Delta t}\mathbf{C}(\tau)d\tau
    \bigg)\mathbf{f}(s)ds
\end{equation} then we have by the explicit flow map (\ref{eqn:time-dep-linear-system-with-force}) that:
\begin{equation}\label{eqn:recurrence}
    \mathbf{v} = 
    \mathbf{M}_{\Delta t}^{(m+1)}\mathbf{u} + 
    \mathbf{g}_{\Delta t}^{(m+1)}
\end{equation} 

Iteratively applying the recurrence (\ref{eqn:recurrence}) to $\mathbf{Y}$, we have the explicit dependence for each column $1\le i\le m$:
\begin{equation}
    \mathbf{y}_i = \mathbf{M}^{(i)}_{\Delta t}\mathbf{x}_i + \mathbf{g}_{\Delta t}^{(i)}
\end{equation} and therefore:
\begin{equation}
    \mathbf{Y} = 
    \begin{bmatrix}
        \vline & \vline & \cdots & \vline \\
        \mathbf{M}_{\Delta t}^{(1)}\mathbf{x}_1+\mathbf{g}_{\Delta t}^{(1)} & \mathbf{M}_{\Delta t}^{(2)}\mathbf{x}_{2}+\mathbf{g}_{\Delta t}^{(2)} & \cdots & \mathbf{M}_{\Delta t}^{(m)}\mathbf{x}_m + \mathbf{g}_{\Delta t}^{(m)} \\
        \vline & \vline & \cdots & \vline
    \end{bmatrix}
\end{equation}

Then we have:
\begin{equation}
    \normtwo{\mathbf{\mathbf{Y}}}_{2}^2 \le 
    \normtwo{
    \begin{bmatrix}
        \mathbf{M}_{\Delta t}^{(1)} & & & \\
        & \mathbf{M}_{\Delta t}^{(2)} & & \\
        & & \ddots & \\
        & & & \mathbf{M}_{\Delta t}^{(m)}
    \end{bmatrix}
    }^2_2\normtwo{\mathbf{X}}_F^2 + 
    \sum_{i=1}^m\normtwo{\mathbf{g}_{\Delta t}^{(i)}}_2^2
\end{equation} To refine the bound, we consider:
\begin{eqnarray*}
    \normtwo{\mathbf{M}_{\Delta t}^{(i)}}_2^2 = 
    \normtwo{
        \exp\bigg(
            \int_{t_i}^{t_i+\Delta t}\mathbf{C}(s)ds
        \bigg)
    }_2^2 \le 
    \exp\bigg(
        \int_{t_i}^{t_i+\Delta t}\normtwo{\mathbf{C}(s)}_2^2ds
    \bigg) 
    \le \exp(\Delta t\gamma_i^2)
\end{eqnarray*} where:
\begin{equation}
    \gamma_i := \max_{t_i\le s\le t_i+\Delta t}\normtwo{\mathbf{C}(s)}_2
\end{equation} 

Under the assumption of piecewise continuity on each $[t_i, t_{i+1}]$, the attainability of $\gamma_i$ is given by considering the spectra of $\mathbf{C}(t)$ as a continuous map of time \cite{DIECI2006502, Bhatia1997}. Furthermore,
\[
    \norm{\mathbf{g}_{\Delta t}^{(i)}}_2^2 = 
    \norm{
        \int_{t_i}^{t_{i}+\Delta t}
        \exp\bigg(
            \int_s^{t_i+\Delta t}\mathbf{C}(\tau)d\tau
        \bigg)\mathbf{f}(s)ds
    }_2^2
\]
\[
    \le f_i^2\int_{t_i}^{t_i+\Delta t}\exp\bigg(
        \gamma_i^2(t_i+\Delta t-s)
    \bigg)ds = \frac{f_i^2}{\gamma_i^2}(\exp(\gamma_i^2\Delta t) - 1)
\] where we define:
\begin{equation}
    f_i := 
    \max_{t_i\le s\le t_{i+1}}\norm{\mathbf{f}(s)}_2
\end{equation} which is attainable due to the piecewise continuous assumption of $\mathbf{f}(t)$. 

Finally, define:
\begin{equation}\label{eqn:define-gamma}
    \gamma := \max_{1\le i\le m}\gamma_i, f := \max_{1\le i\le m}f_i
\end{equation} 

We conclude the following result as desired:
\[
    \norm{\mathbf{Y}}_2^2 
    \le 
    \exp(\gamma^2\Delta t)\sum_{i=1}^{m}\sigma_i^2(\mathbf{X}) + \frac{mf}{\gamma}(\exp(\gamma^2\Delta t)-1) 
    \\
\le \exp(\gamma^2\Delta t)\bigg(\frac{mf^2}{\gamma^2} +
    \sum_{i=1}^{m}\sigma_i^2(\mathbf{X})\bigg)
\]
\end{proof}

\begin{remark}
    In the special case where $\mathbf{C}(t)\equiv\mathbf{C}$ with eigendecomposition $\mathbf{C} = \mathbf{Q}\boldsymbol{\Lambda}\mathbf{Q}^{-1}$ and largest eigenvalue $\lambda_1$, and $\mathbf{f}\equiv\mathbf{0}$. We have that the solution has the form:
    \begin{equation}
        \mathbf{x}(t) = \mathbf{Q}\exp(t\boldsymbol{\Lambda})\mathbf{Q}^{-1}\mathbf{x}_0
    \end{equation} 

Under the same conditions, the upper bound in Proposition~\ref{lemma:time-shift-norm-bounds} can be tightened to:
\[
        \norm{\mathbf{Y}}_2 \le \kappa_2(\mathbf{Q})\exp\big(\lambda_1\Delta t\big)\sigma_{max}(\mathbf{X})
\] where $\kappa_2(\cdot)$ denotes the 2-norm condition number.
\end{remark}

We provide a verification of the upper bounds in Proposition~\ref{lemma:time-shift-norm-bounds} in Figure~\ref{fig:time-shift-upper-bound} using the time-varying linear system of \cite{zhang2017online}, Example 5.2:
\begin{equation} \label{eqn:simple-linear-system}
    \frac{d\mathbf{x}(t)}{dt} = \mathbf{C}(t)\mathbf{x}(t)
\end{equation}
\[
    \mathbf{x}(0) = [1, 0]^T
\] where:
\begin{equation}
    \mathbf{C}(t) = 
    \begin{bmatrix}
        0 & 1+\epsilon t\\
        -1-\epsilon t & 0
    \end{bmatrix}
\end{equation} with $\epsilon = 0.1$ on the temporal domain $t\in [0,1]$, with $\Delta t=10^{-3}$. Furthermore, we also provide the upper bounds for the two advection-dominated examples with the parameter setups described in Section~\ref{eqn:1d-advection} and Section~\ref{eqn:2d-advection}. In particular, the example system (\ref{eqn:time-dep-linear-system-with-force}) is especially useful for the consideration of numerical solutions to the linear PDE (\ref{eqn:variable-advection-diffusion}), where the matrix $\mathbf{C}(t)$ may be seen as the finite difference or finite element stiffness matrix with time-varying coefficients, and $\mathbf{f}(t)$ as the inhomogeneous source term.

\begin{figure}
    \centering
    \includegraphics[width=8cm]{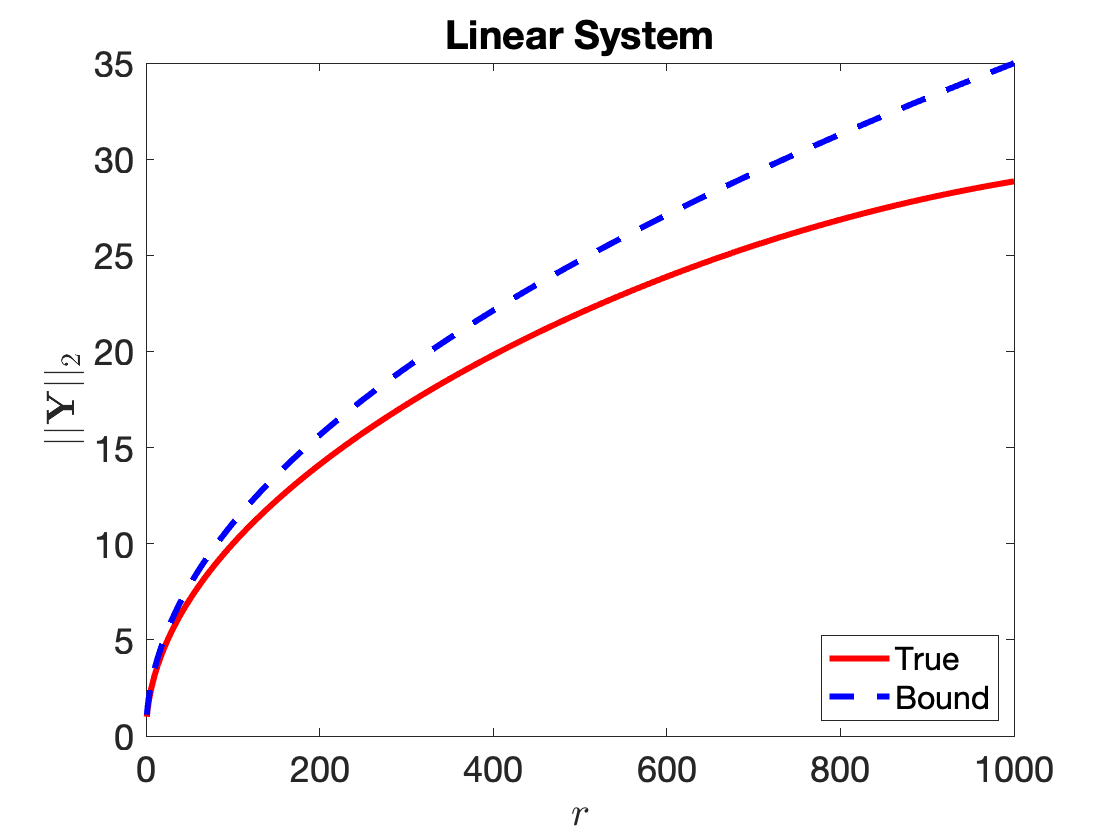}
    \includegraphics[width=8cm]{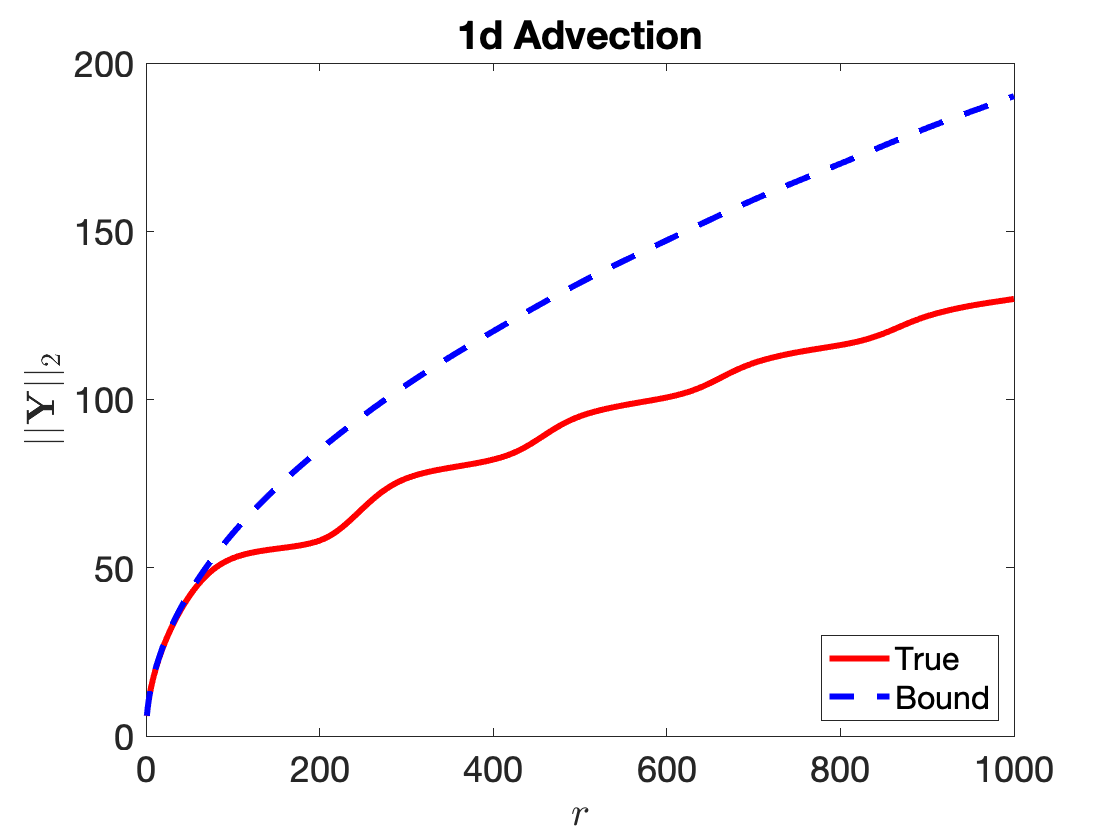}
    \includegraphics[width=8cm]{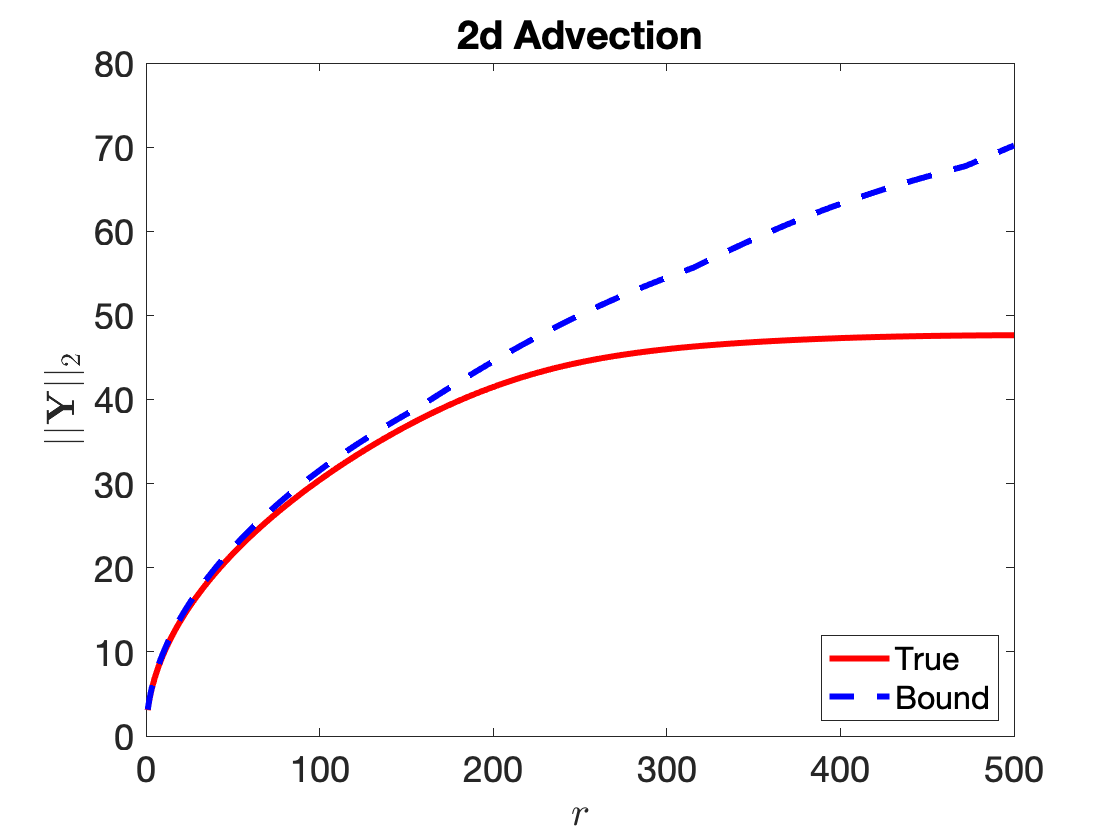}
    \caption{Time shift data matrix 2 norm upper bounds (\ref{lemma:time-shift-norm-bounds}) compared to actual 2 norms, with respect to number of collected snapshots. Top: linear system (\ref{eqn:simple-linear-system}) with $N=2, \Delta t=10^{-3}$. Middle: time-varying advection in 1d (\ref{eqn:advection1d}) with $N=400, \Delta t=0.01$. Bottom: time-varying advection-diffusion in 2d (\ref{eqn:2d-advection-equation}) with $N=2500$ and $\Delta t=0.01$.  }
    \label{fig:time-shift-upper-bound}
\end{figure}

The interpretations of the results obtained in Proposition~\ref{eqn:op-norm-bound-truncating-column} and Proposition~\ref{lemma:time-shift-norm-bounds} are two-folds. In cases where the learning is agnostic of underlying physics (i.e. with data only available as images and the underlying system is unknown), such as the cases considered in \cite{grosek2014dynamic}, perturbations in the DMD operator will strictly be estimable as the perturbation in collected data snapshots alone. However, with additional information of the underlying system, such as (\ref{eqn:time-dep-linear-system-with-force}), one may incorporate physical knowledge and refine the bound by considering columns of $\mathbf{X}, \mathbf{Y}$ as ordered time-shifts of the initial condition along the flow map. Nevertheless, both of the results serve as a priori estimates of operator norm perturbation to help guide the selection of hyperparameters in DMD algorithms. 

\section{Numerical Experiments} \label{sec:numerical-examples} In the following numerical examples, we test the accuracy of Algorithm~\ref{alg:time-dependent-lagrangian-dmd} for a variety of time-varying advection phenomena. In particular, for advection-dominated linear conservation laws (Sections \ref{eqn:1d-advection} and \ref{eqn:2d-advection}), we make the procedure fully data-driven by assuming that the advection velocity in equation (\ref{eqn:lagrangian-pde}) is unknown, and is estimated from tracking the trajectory of the mode. 

Given a temporal discretization $0 = t_0 < t_1 < \ldots, < t_n = t_f$, we measure the performance of DMD algorithms via relative prediction error defined as the following:
\begin{equation} \label{eqn:define-relative-error}
    \epsilon(t) =\frac{\norm{\mathbf{u}_{\text{DMD}}(t) - \mathbf{u}(t)}_2}{\norm{\mathbf{u}(t)}_2}
\end{equation} where $\mathbf{u}, \mathbf{u}_{\text{DMD}}$, are respectively the exact solution and the DMD prediction at time $t$, with the error computed in the $L^2(\mathbb{R}^{d})$ sense. To construct the reduced order model in each experiment, an SVD and projection to POD modes are applied at a prespecified rank determined based on a relative accuracy tolerance level. The exact setup of each numerical simulations is reported separately. 

We first consider the Navier-Stokes equations to test the accuracy of the base time-varying DMD algorithm in reconstructing complex and nonlinear dynamics without Lagrangian information, presented in Section~\ref{eqn:2d-coupled}. For each experiment of Section~\ref{eqn:1d-advection} and \ref{eqn:2d-advection}, we compare four different strategies: the standard DMD and time-varying DMD using only $\mathbf{u}(t)$ as observables, the physics-aware DMD in Section~\ref{sec:lagrangian-dmd-section} without recomputations, and Algorithm~\ref{alg:time-dependent-lagrangian-dmd}. 

\subsection{Incompressible Navier-Stokes equations} \label{eqn:2d-coupled} We consider the flow field of a two-dimensional incompressible fluid with density $\rho = 1$ and dynamic viscocity $\nu = 1/600$ (kg/(m$\cdot$s)). With a rectangular domain $\mathcal{D} = [0,2]\times [0,1]$, the fluid enters from the left boundary with fixed velocity and flows around an impermeable cylinder centered at $\mathbf{x}_{\text{circ}} = [0.3, 0.5]^T$. The dynamics of fluid pressure $p(t,\mathbf{x})$, horizontal velocity component $u(t,\mathbf{x})$ and vertical velocity component $v(t,\mathbf{x})$ follow the Navier-Stokes (NS) equation:
\begin{equation} \label{eqn:define-ns-equation}
    \begin{dcases}
        \frac{\partial u}{\partial t} + u\frac{\partial u}{\partial x} + v\frac{\partial u}{\partial y} = -\frac{1}{\rho}\frac{\partial p}{\partial x} + \nu\bigg(
        \frac{\partial^2 u}{\partial x^2} + \frac{\partial^2 u}{\partial y^2}
    \bigg) \\
        \frac{\partial v}{\partial t} + u\frac{\partial v}{\partial x} + v\frac{\partial v}{\partial y} = -\frac{1}{\rho}\frac{\partial p}{\partial y} + \nu\bigg(
        \frac{\partial^2 v}{\partial x^2} + \frac{\partial^2 v}{\partial y^2}
    \bigg) \\
    \frac{\partial u}{\partial x} + \frac{\partial v}{\partial y} = 0
    \end{dcases} 
\end{equation} subject to the following initial-boundary conditions:
\begin{equation} \label{eqn:navier-ibconditions1}
    p(t, 2, y) = 0, \frac{\partial p}{\partial \mathbf{n}}\big\rvert_{\partial\mathcal{D}\setminus \{x=2\}} = 0, \frac{\partial u(t, 2, y)}{\partial \mathbf{n}} = 0, \frac{\partial v(t, 2, y)}{\partial \mathbf{n}} = 0
\end{equation}
\begin{equation} \label{eqn:navier-ibconditions2}
    u(t, 0, y) = 1, v(t,0,y) = 0, 
\end{equation}
\begin{equation} \label{eqn:navier-ibconditions3}
    u(t, x, 0) = u(t, x, 1)= 0, v(t, x, 0) = v(t, x, 1) = 0
\end{equation}

We define the quantity of interest as the magnitude of our velocity field:
\begin{equation} \label{eqn:2d-navier-qoi}
    w(t, x, y) := \sqrt{u(t,x,y)^2 + v(t,x,y)^2}
\end{equation} and simulate the nonlinear system (\ref{eqn:define-ns-equation}) with a custom MATLAB library. The problem is solved in conservative form with finite difference method on a staggered grid \cite{Danggo_2017}, with discretization levels $\Delta x = \Delta y = 0.02$, and time step size $\Delta t = 0.001$.

Under this setting, we focus on reconstructing the dynamics during formation of vortex street on the time domain $t\in [0, 3.0]$, yielding effective state dimensions $N=5000$ and $m=3000$ snapshots. For each DMD strategy, we set the SVD truncation level to $\epsilon=1.0\times 10^{-2}$. Figures~\ref{fig:2d-navier-compare1} and Figure~\ref{fig:2d-navier-error} shows a comparison of predicted solutions between standard DMD and time-varying DMD along with their relative $L^2$-errors from the reference numerical solution. As expected, standard DMD places an invariant manifold assumption and yields an inaccurate reduced-order model under rapid time changes. The time-varying DMD more accurately represents the solution by updating the operator at different time intervals. Finally, we visualize the dominant frequency variations during the time domain $[0, 2.5]$ and observe that standard DMD begins to accumulate errors after $t=0.05$, failing to capture the rapid frequency changes.

\begin{figure}
    \centering
    \includegraphics[width=4.2cm]{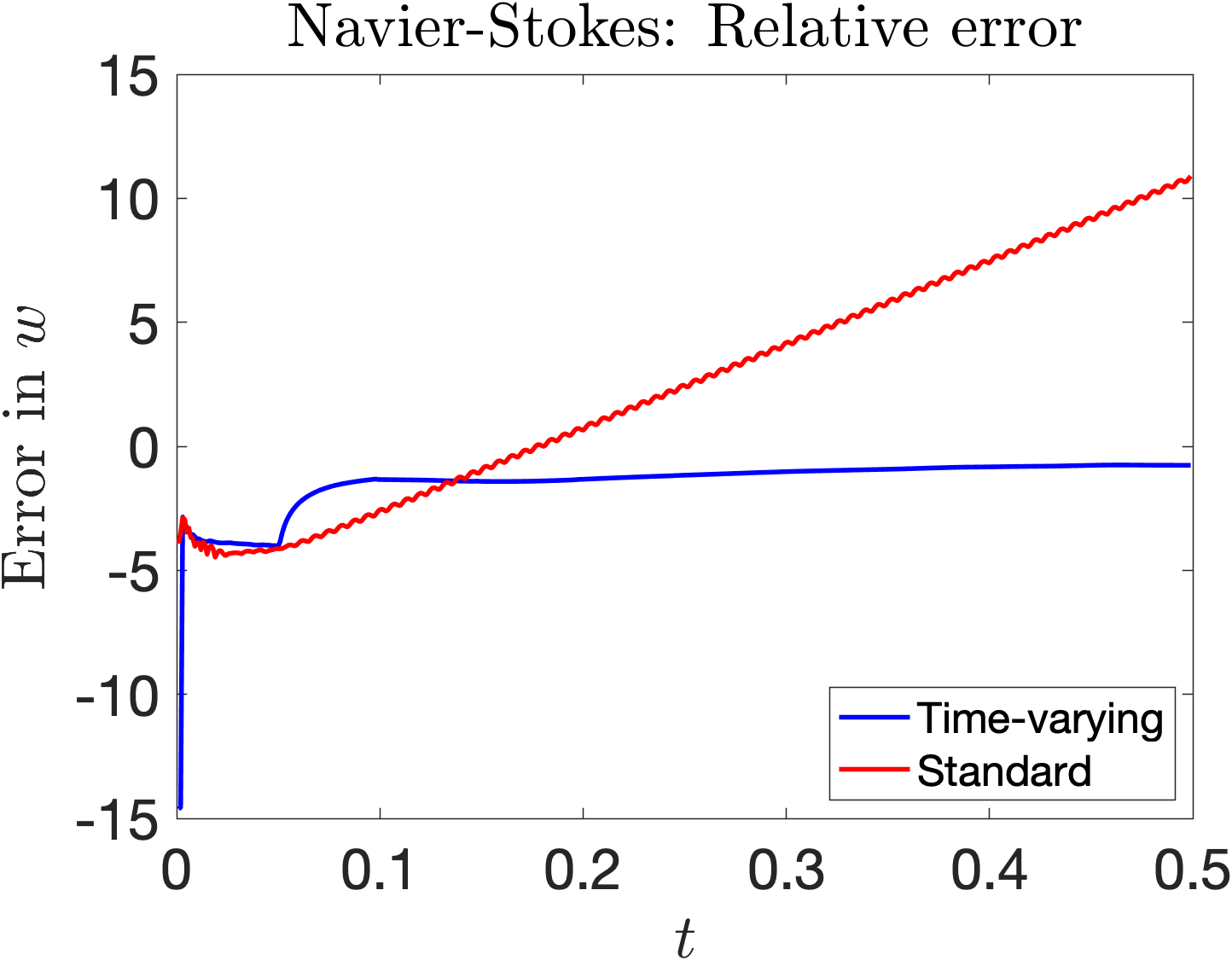}
    \includegraphics[width=4.2cm]{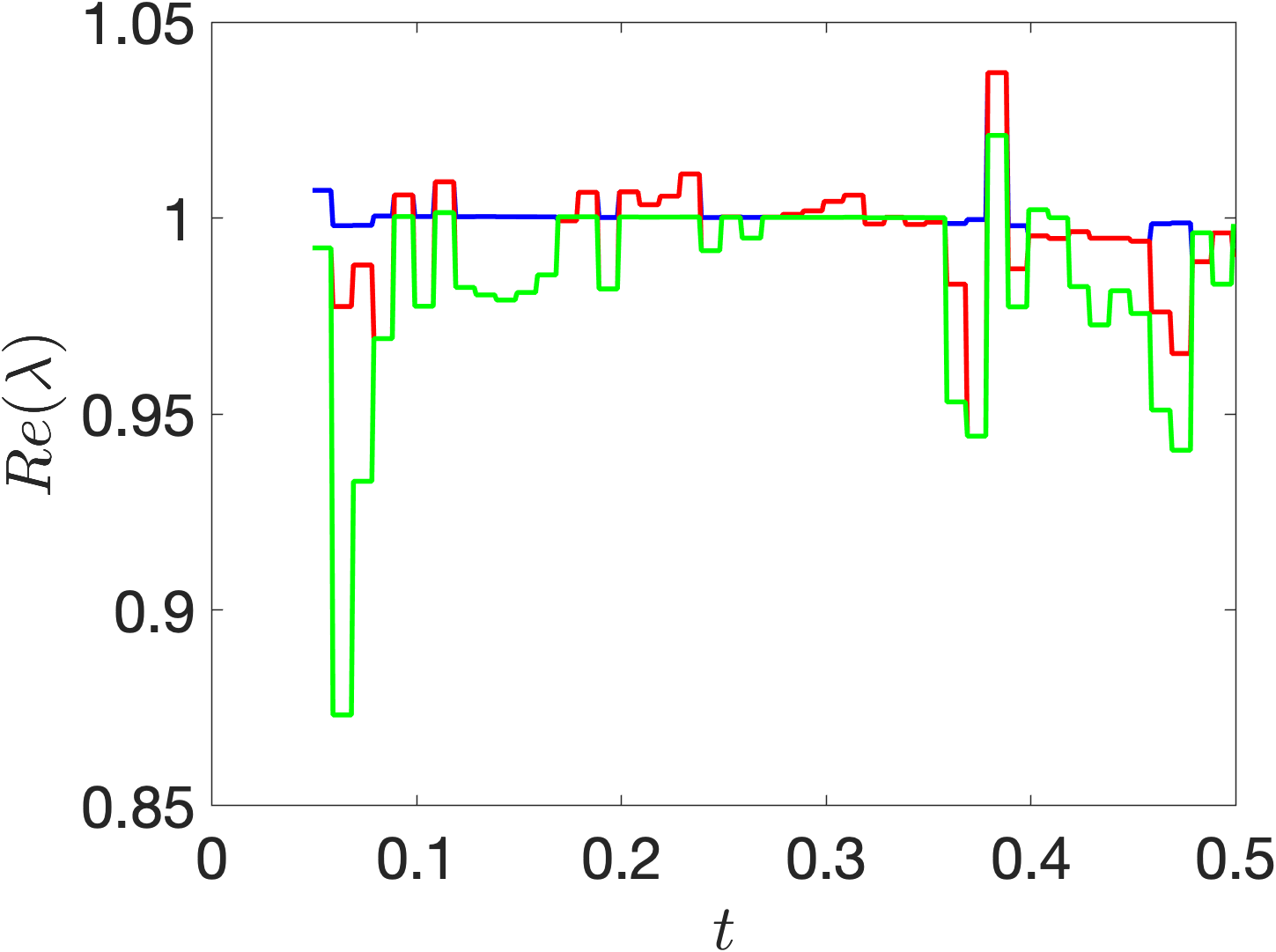}
    \includegraphics[width=4.2cm]{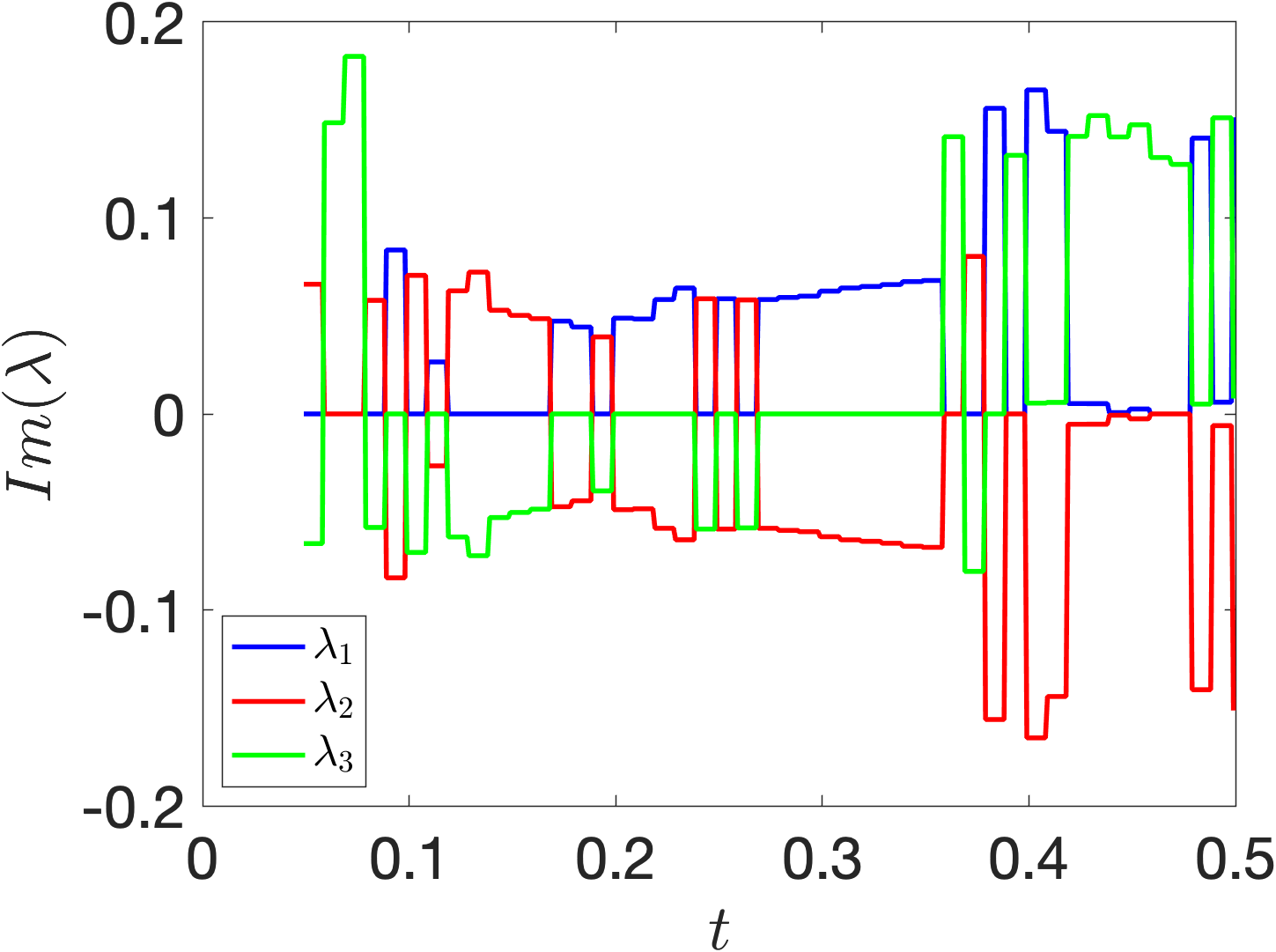}
    \caption{Left: comparison of standard DMD and time-varying DMD in terms of prediction relative errors. Middle: real part of top 3 dominant frequencies, computed from time-varying DMD modes,  as a function of time. Right: imaginary part of top 3 dominant frequencies as a function of time.}
    \label{fig:2d-navier-error}
\end{figure}

\begin{figure}
    \centering
    \includegraphics[width=13cm]{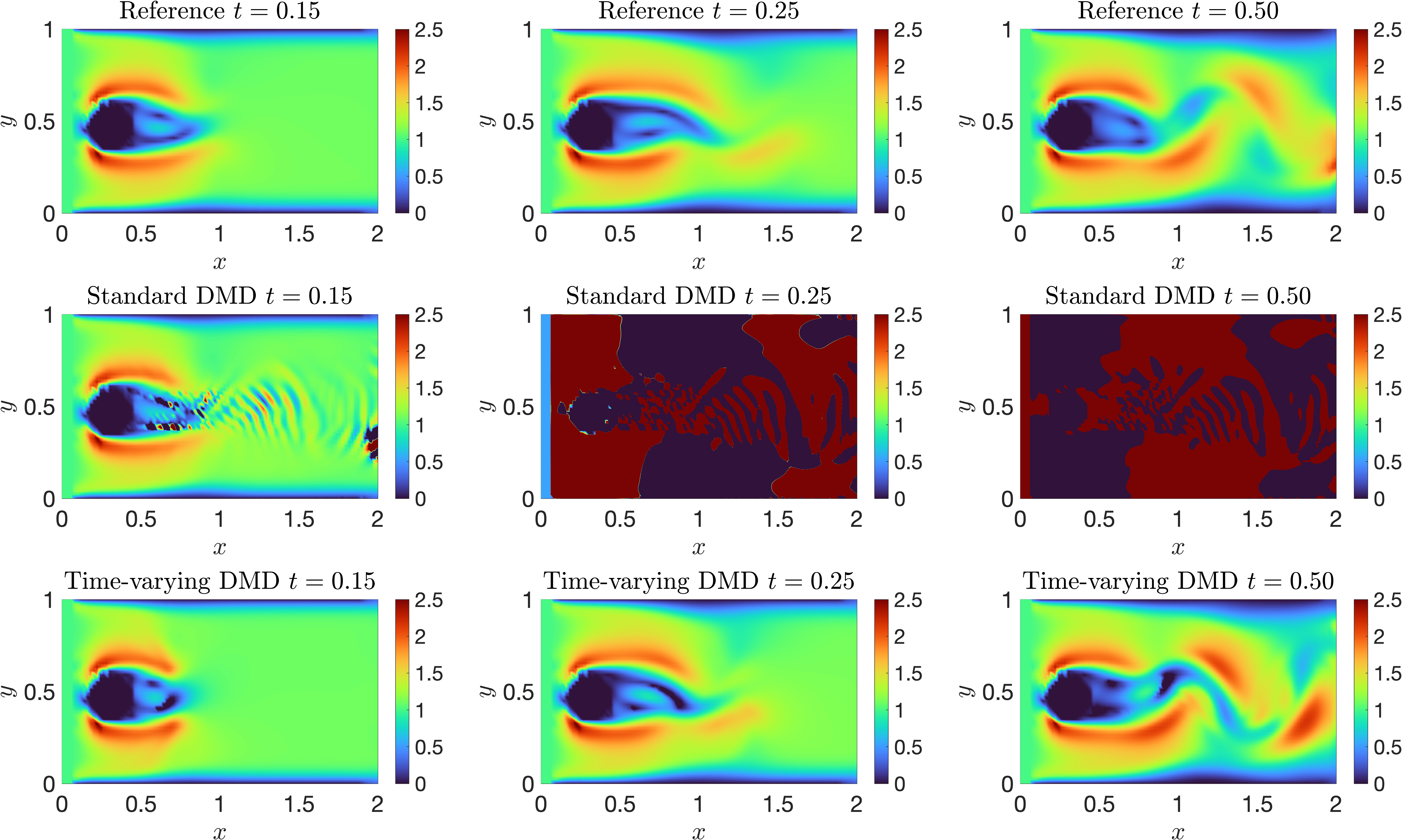}
    \caption{Reconstructed velocity magnitudes to the 2d Navier-Stokes equation (\ref{eqn:2d-navier-qoi}) at time steps $t=0.15, 0.25, 0.5$. Top row: reference solution from high-fidelity simulation. Middle row: standard DMD predictions. Bottom row: time-varying DMD predictions ($r=50$). }
    \label{fig:2d-navier-compare1}
\end{figure}

\subsection{1d time-varying advection} \label{eqn:1d-advection} As a test problem for a comprehensive comparison understanding of standard DMD, time-varying DMD (without Lagrangian information), physics-aware DMD (without temporal updates), and time-varying DMD with Lagrangian moving grid information, we consider the following conservation law under pure advection $(D\equiv 0)$:
\begin{equation}\label{eqn:advection1d}
    \begin{dcases}
        \frac{\partial u}{\partial t} + \frac{\partial}{\partial x}[c\sin(\omega t)u] = 0\\
        u(0, x) = u_0(x) = \exp(-0.5x^2)
    \end{dcases} 
\end{equation} where we choose the advection speed $c\equiv 2$ and frequency $\omega \equiv \pi/2$. The snapshots are simulated using an upwind numerical scheme on a temporal grid of $t\in [0, 8]$ with discretization $\Delta t = 0.01$, yielding $m=800$ training data points. The spatial grid is taken to be $x\in [-10, 10]$ with discretization $\Delta x = 0.05$. By construction, the initial concentration $u_0$ does not change shape, and is advected in a oscillatory manner over time. As a fully data driven model, we consider estimating the velocity as a function of time directly from observations. Figure~\ref{fig:1d-advection-velocity} shows a visualization of the advection velocity as a function of time, estimated from tracking the mode of the solution, defined by viewing the conserved solution $u$ as a density, and computing the average:
\begin{equation} \label{eqn:define-mean}
    \overline{x}(t) := \frac{1}{\int_{x_l}^{x_r}u(t,x)dx}\int_{x_l}^{x_r}xu(t,x)dx
\end{equation} where for (\ref{eqn:1d-advection}), $x_l = -10, x_r = 10$. Then the estimated velocity can be computed using a centered difference of $\overline{x}(t)$ at discrete time points, which was then used as an approximation to the velocity in the Lagrangian reference frame of (\ref{eqn:lagrangian-pde}).
\begin{figure}[h!]
    \centering
    \includegraphics[width=6cm]{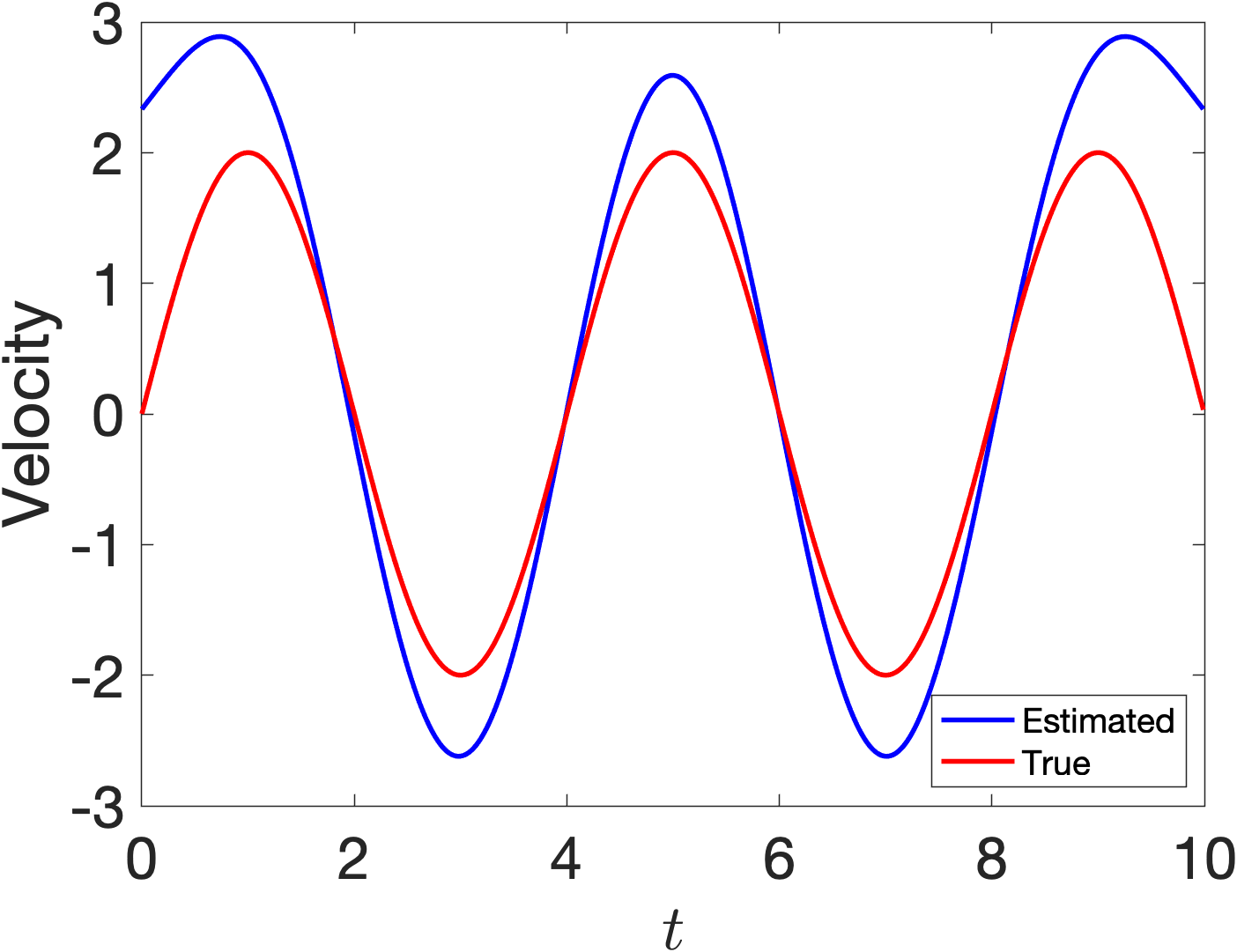}
    \caption{Estimated advection velocity for (\ref{eqn:advection1d}) by tracking the mode of numerical solutions on the time domain $[0, 10]$. }
    \label{fig:1d-advection-velocity}
\end{figure}

We present the predicted solutions, compared with the reference numerical solution, at time steps $t=0, \pi/4, \pi/2, \pi$. In this experiment, we set the tolerance for SVD truncation for all DMD strategies to be $\epsilon=10^{-6}$. Furthermore, for time-varying DMD stratgies, the size of the subintervals are chosen to be $r=5$.

Figures~\ref{fig:1d-adv-standard-dmd}, \ref{fig:1d-adv-standard-incremental-dmd}, \ref{fig:1d-adv-lagrangian-dmd}, and \ref{fig:1d-adv-lagrangian-incremental-dmd} show the behavior of predicted solutions under different DMD strategies. The relative errors are plotted on log scale, presented in Figure~\ref{fig:1d-adv-all-errors}. In particular, we observe increased error fluctuations for time-homogeneous DMD strategies (i.e. standard DMD and physics-ware DMD) at regions of high velocity speed. The advection of the solution mode is also not captured. This is to be expected as standard DMD and physics-aware DMD are assumed to be constant in time, and would incur larger errors where such dependence is stronger. In the case of the time-varying DMD without Lagrangian information, we observe that the modal information is captured and advects through time. However, unphysical oscillations are still present. Out of the tested DMD strategies, Algorithm~\ref{alg:time-dependent-lagrangian-dmd} provides the most faithful reconstruction of the time-varying advection behavior. 

\begin{figure}[h!]
    \centering
    \includegraphics[width=12cm]{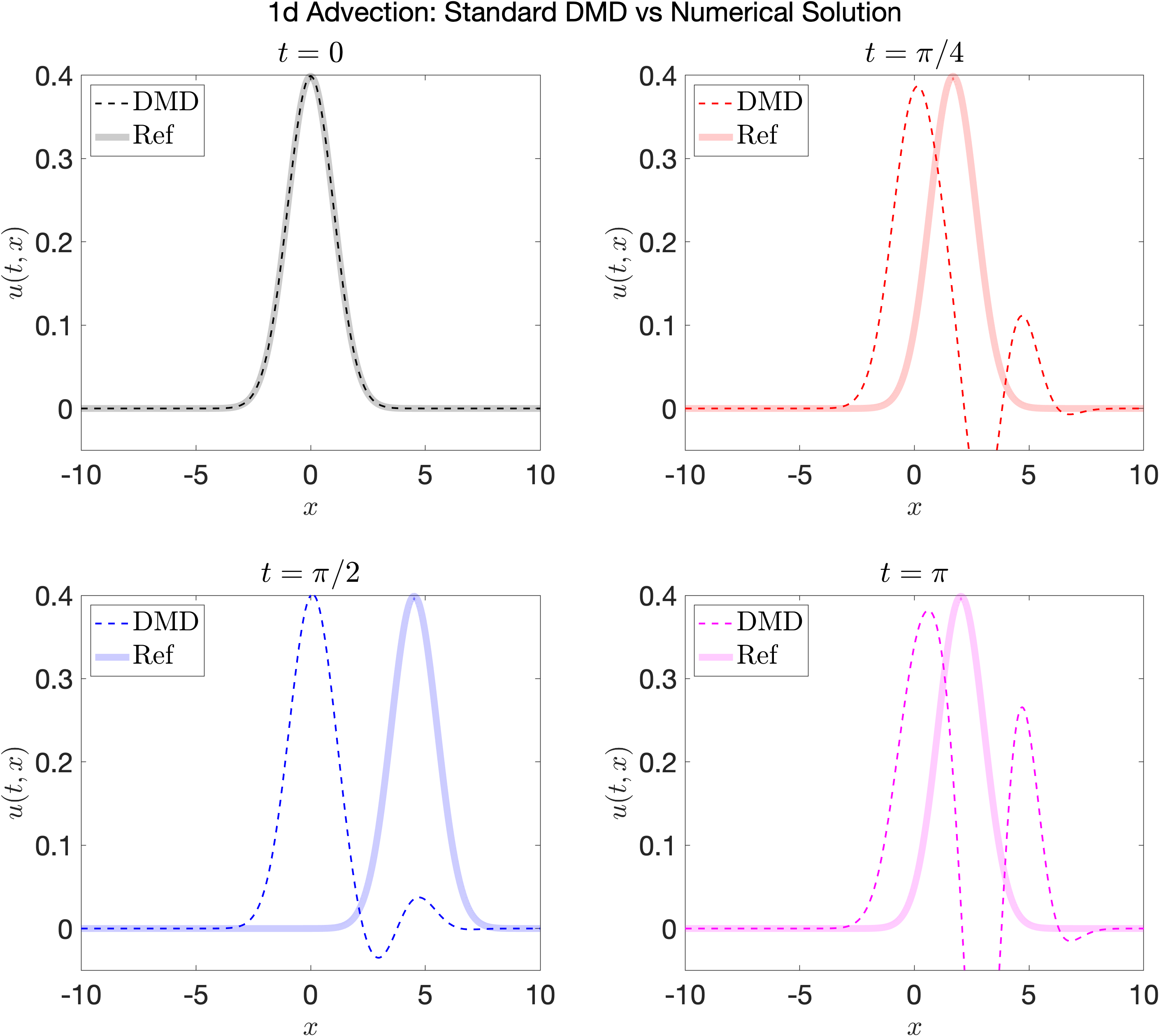}
    \caption{1d time-varying advection: standard DMD predictions at $t=0$, $t=\pi/4$, $t=\pi/2$, $t=\pi$.}
    \label{fig:1d-adv-standard-dmd}
\end{figure}
\begin{figure}[h!]
    \centering
    \includegraphics[width=12cm]{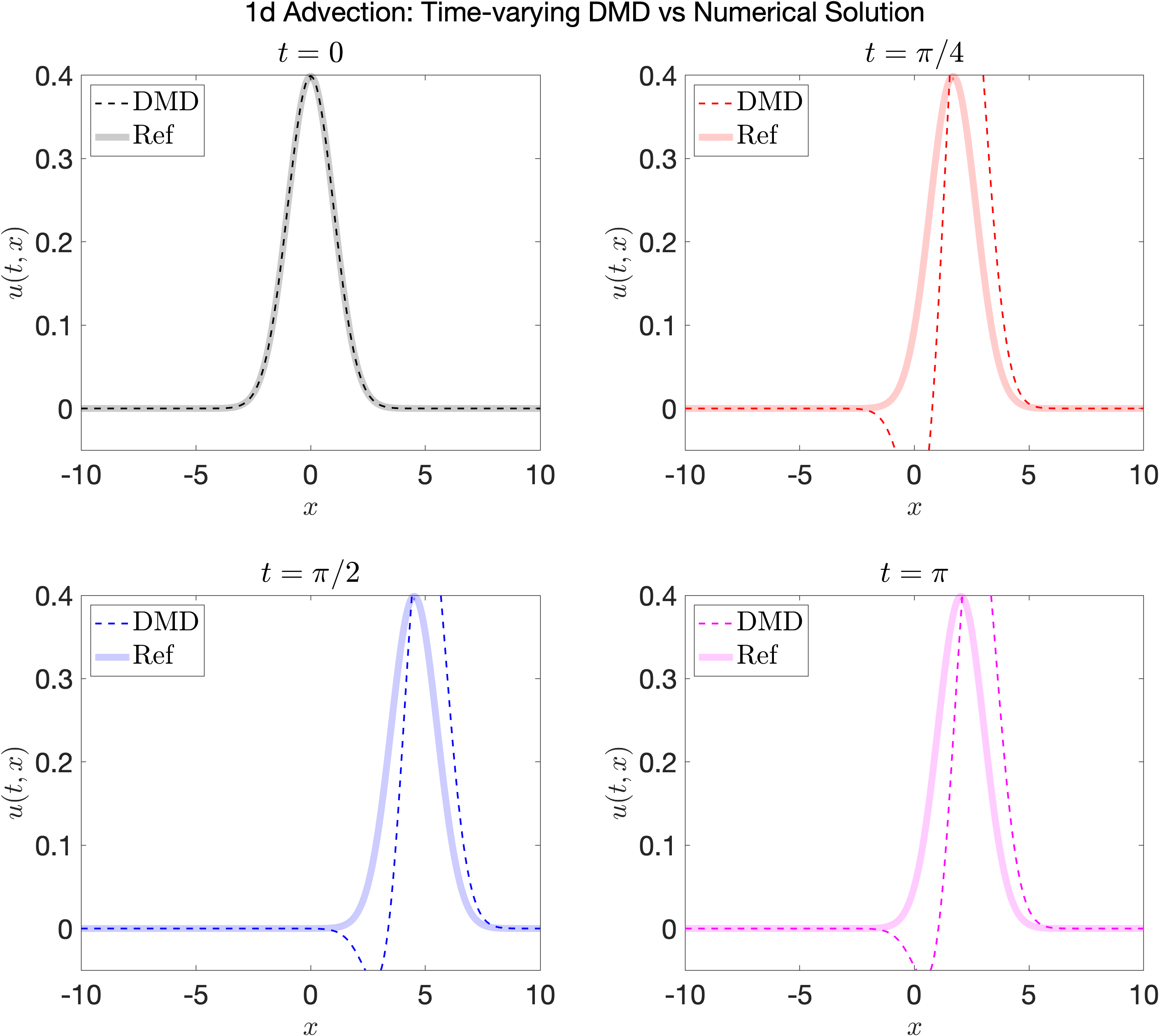}
    \caption{1d time-varying advection: time-varying DMD predictions ($r=5$, without Lagrangian grid), at $t=0$, $t=\pi/4$, $t=\pi/2$, $t=\pi$. }
    \label{fig:1d-adv-standard-incremental-dmd}
\end{figure}
\begin{figure}[h!]
    \centering
    \includegraphics[width=12cm]{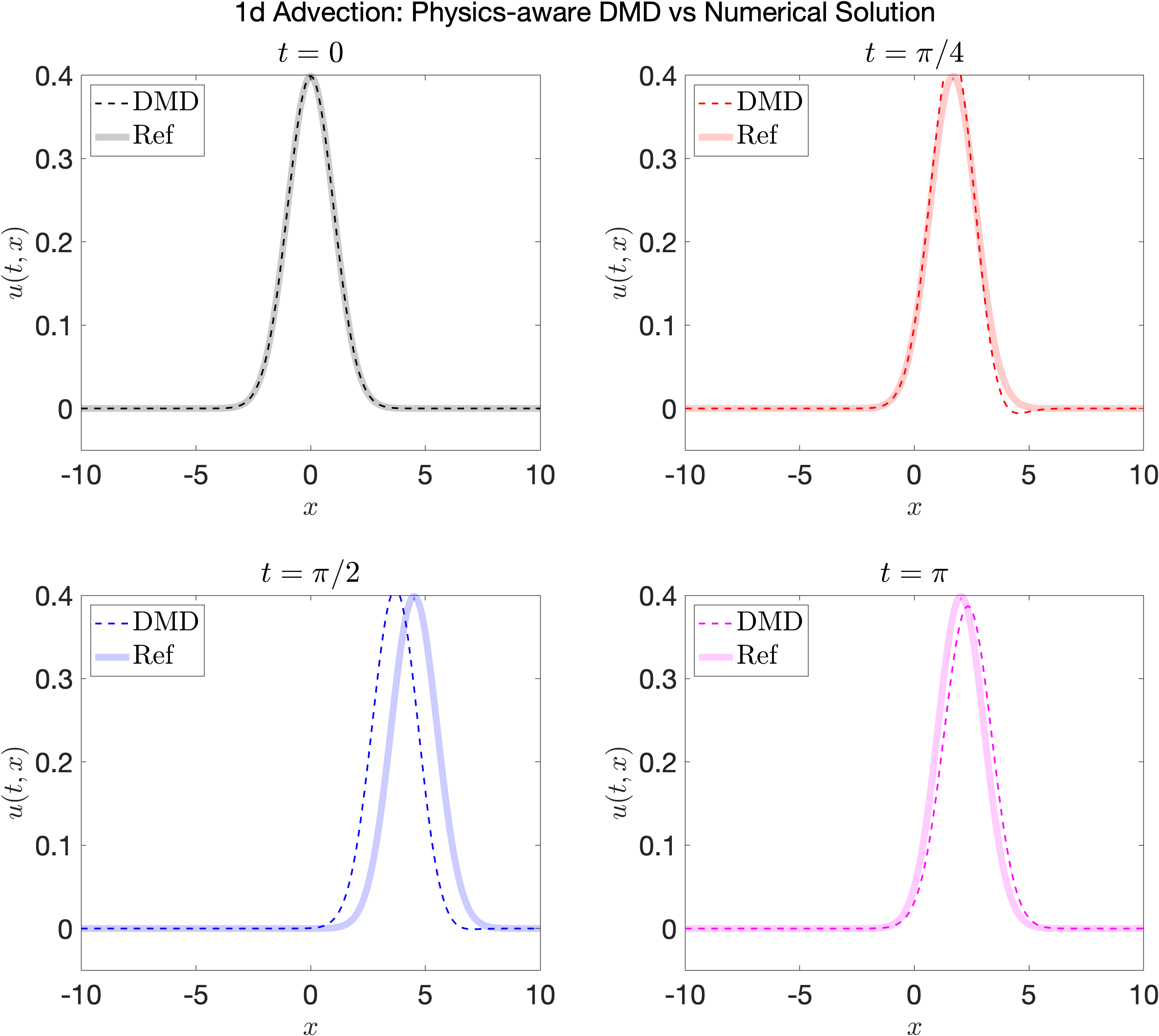}
    \caption{1d time-varying advection: physics-aware DMD predictions at $t=0$, $t=\pi/4$, $t=\pi/2$, $t=\pi$. }
    \label{fig:1d-adv-lagrangian-dmd}
\end{figure}
\begin{figure}[h!]
    \centering
    \includegraphics[width=12cm]{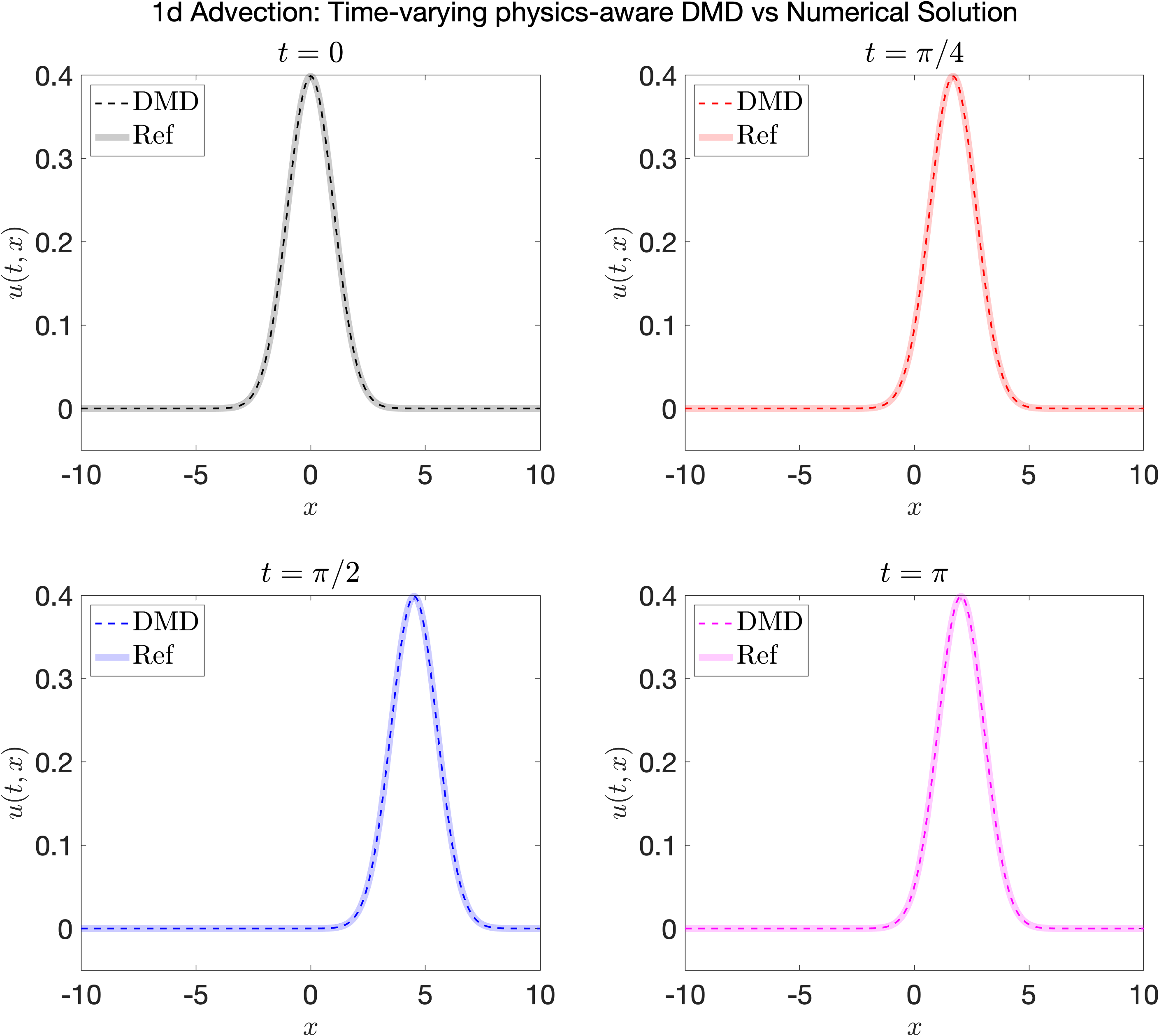}
    \caption{1d time-varying advection: time-varying DMD predictions ($r=5$, with Lagrangian grid), at $t=0$, $t=\pi/4$, $t=\pi/2$, $t=\pi$. }
    \label{fig:1d-adv-lagrangian-incremental-dmd}
\end{figure}
\begin{figure}[h!]
    \centering
    \includegraphics[width=8cm]{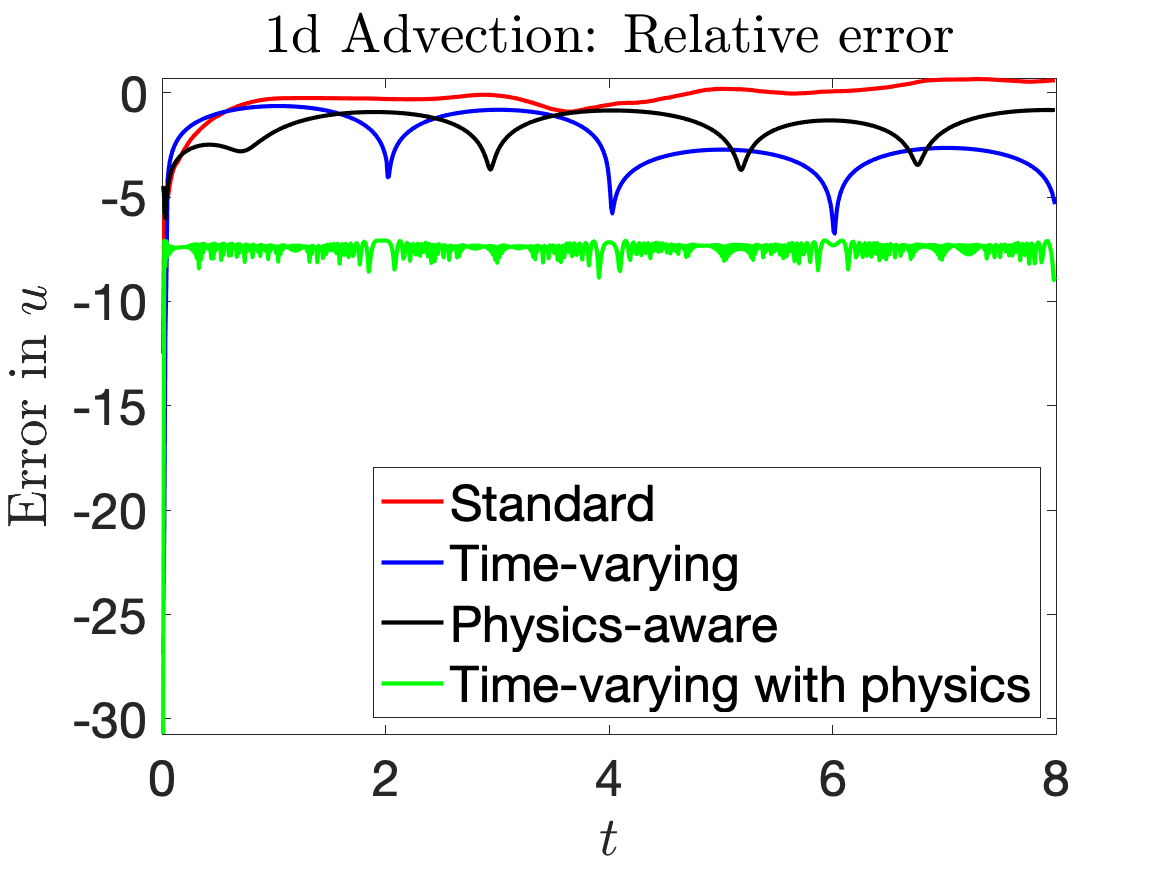}
    \caption{1d time-varying advection: log-scale comparison of $L^2$ relative prediction errors. }
    \label{fig:1d-adv-all-errors}
\end{figure}

\subsection{Advection-dominated Equation in 2d} \label{eqn:2d-advection}  

We consider a two-dimensional linear advection-diffusion equation with time-varying velocity components, defined on the spatio-temporal domain: $(t, x, y) \in [0, 10]\times [-10, 10]\times [-10, 10]$.
\begin{equation} \label{eqn:2d-advection-equation}
    \begin{dcases}
        \frac{\partial u}{\partial t} + v_x(t)\frac{\partial u}{\partial x} + v_y(t)\frac{\partial u}{\partial y} = D\bigg(
            \frac{\partial^2u}{\partial x^2} + \frac{\partial^2u}{\partial y^2}
        \bigg) \\
        u(0, x, y) = \exp(-(x^2 + y^2)) \\
        v_x(t) = \frac12\cos(t), v_y = -\frac{2}{5}\sin(t), D \equiv 0.001
    \end{dcases}    
\end{equation} 

In this example, we let the number of spatial nodes in each direction be $N_x = N_y = 50$. The temporal domain is discretized with a step size of $\Delta t = 0.01$. The PDE is numerically solved using a modified centered time, centered space method (Du-Fort Frankel method) presented in \cite{Hutomo_2019}. The above discretization yields state dimension $N=2500$ and number of snapshots $M = 1000$. 

\begin{figure}[hbt!]
    \centering
    \includegraphics[width=12cm]{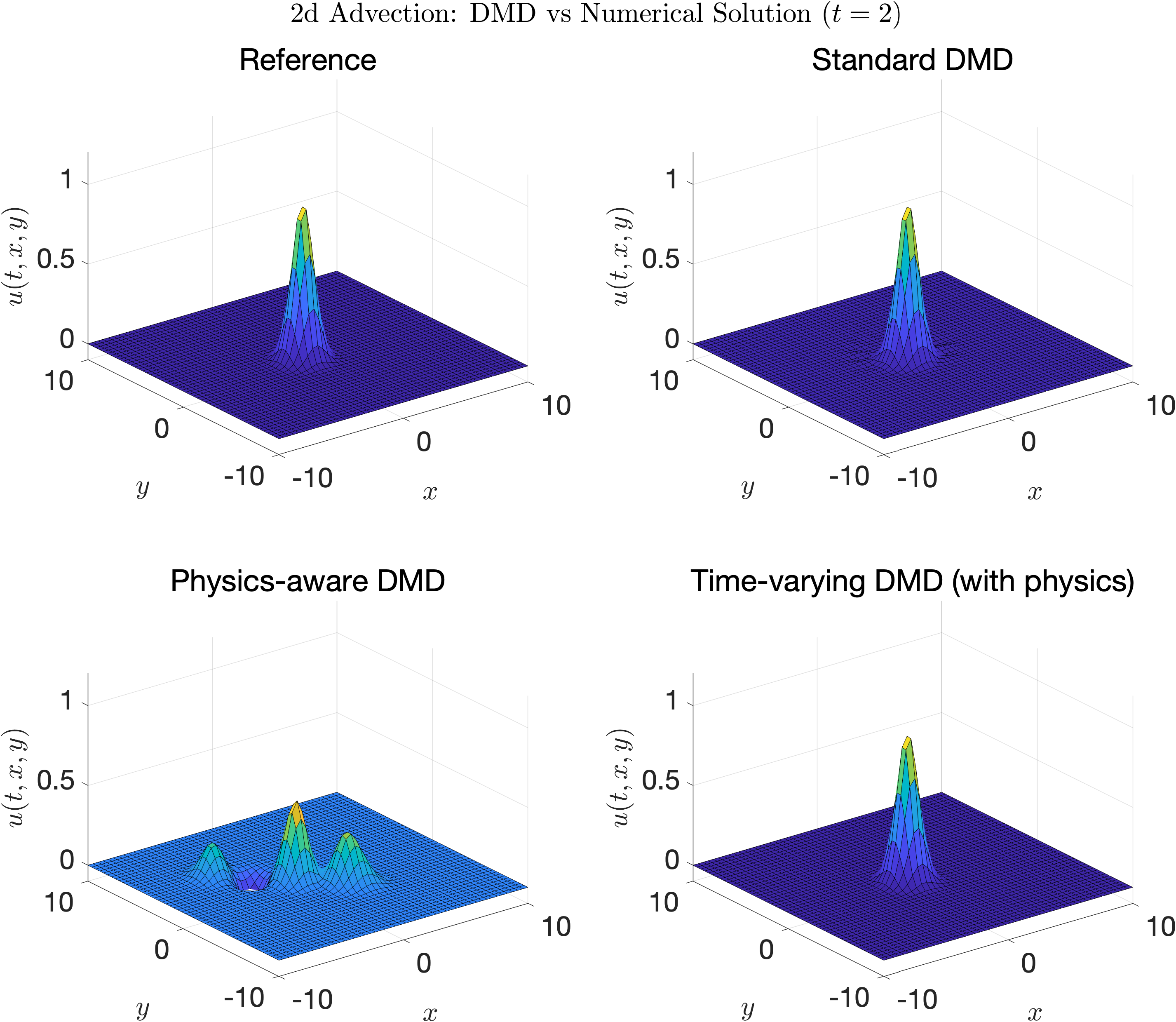}
    \caption{2d time-varying advection-diffusion: comparison of DMD predictions at $t=2$.}
    \label{fig:2d-adv-sol1}
\end{figure}
\begin{figure}[hbt!]
    \centering
    \includegraphics[width=12cm]{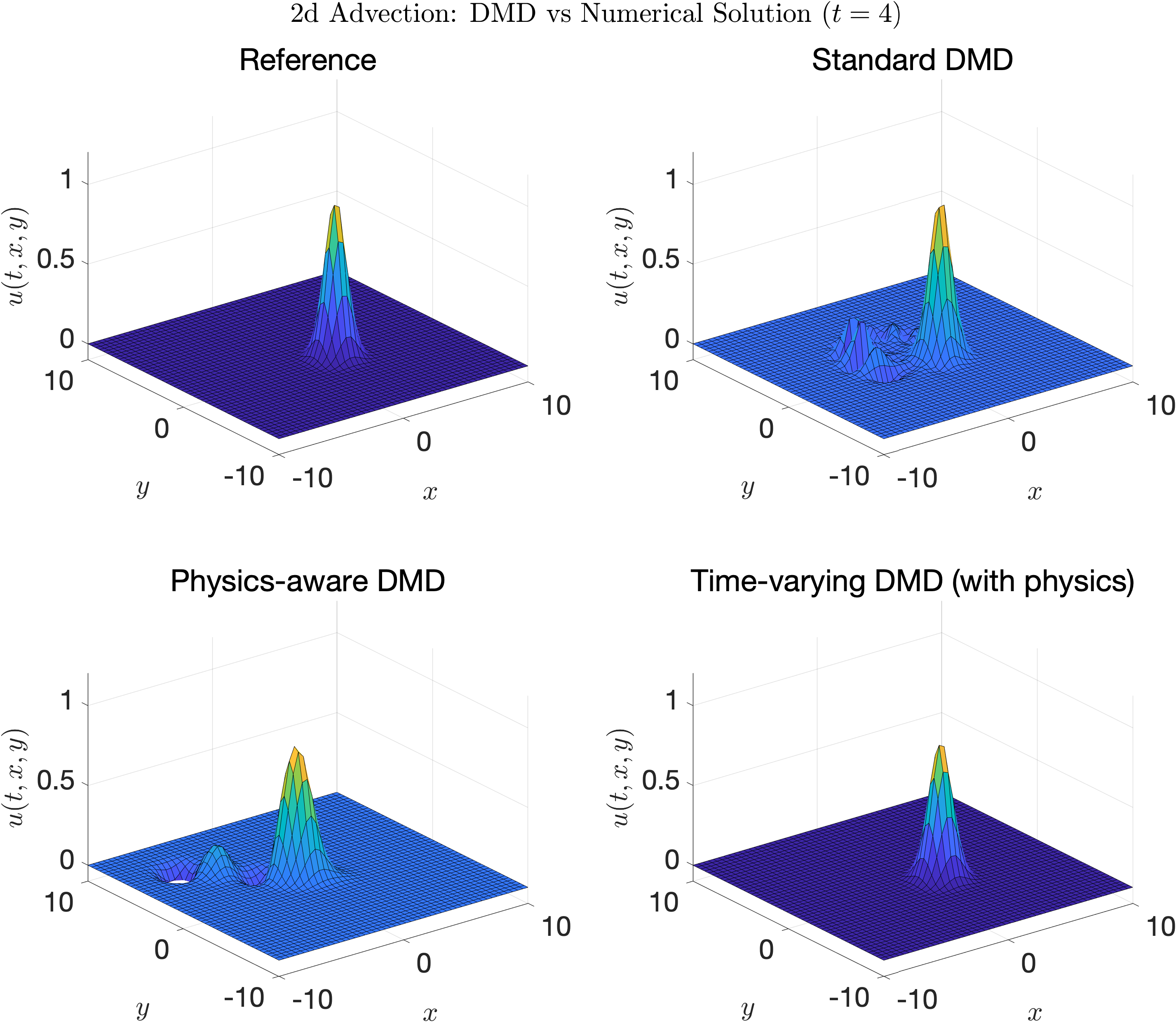}
    \caption{2d time-varying advection-diffusion: comparison of DMD predictions at $t=4$.}
    \label{fig:2d-adv-sol2}
\end{figure}
\begin{figure}[hbt!]
    \centering
    \includegraphics[width=8cm]{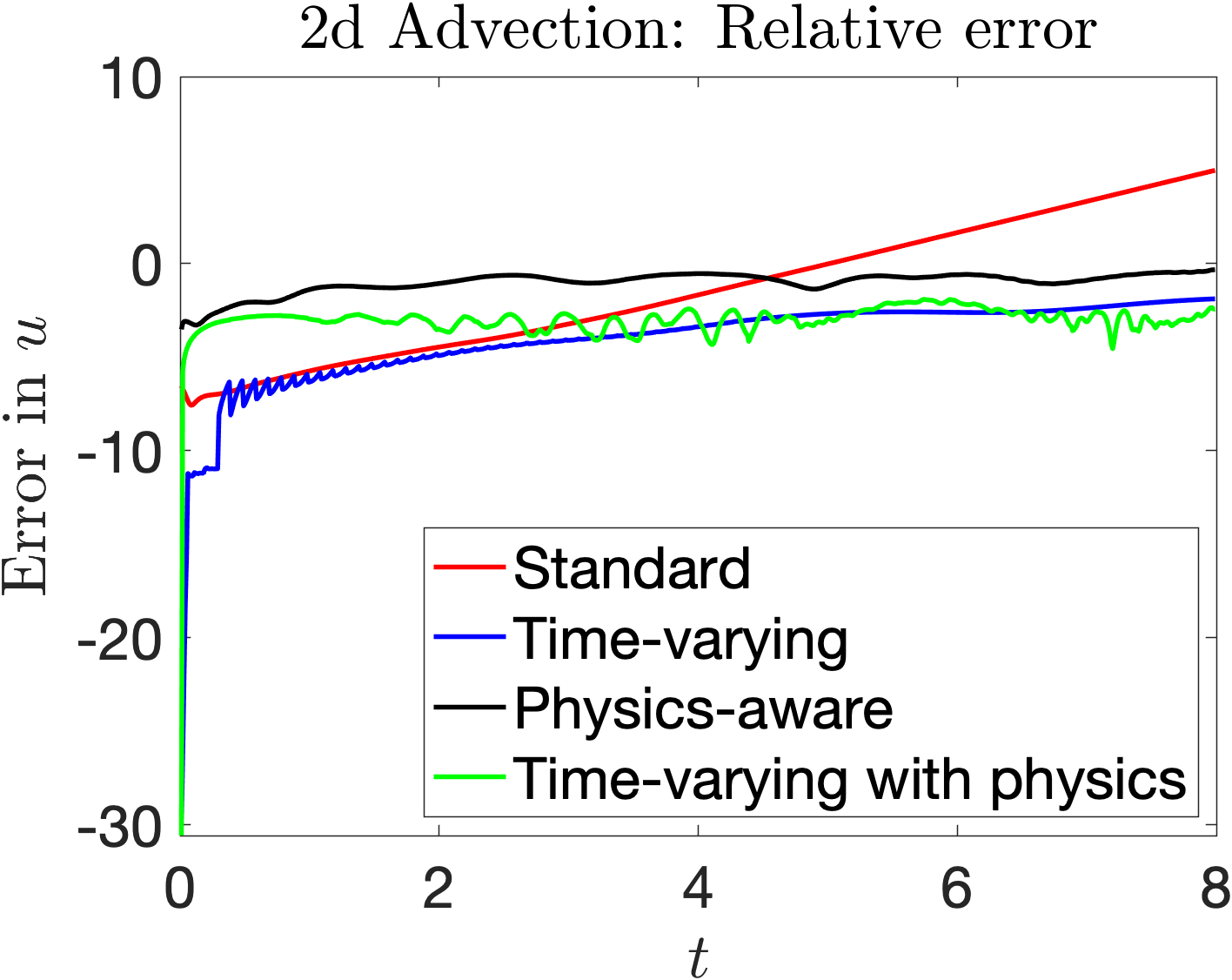}
    \caption{2d time-varying advection-diffusion: log-scale comparison of $L^2$ relative prediction errors.}
    \label{fig:2d-adv-error}
\end{figure}

Similar to the 1-dimensional problem (\ref{eqn:advection1d}), the advection velocity can be estimated in a fully data-driven manner by tracking the mode of the solution snapshots by defining, analogously to (\ref{eqn:define-mean}):
\begin{equation}\label{eqn:define-mean-2d}
    \overline{\mathbf{x}}(t) := \begin{bmatrix} \overline{x}(t) \\ \overline{y}(t)\end{bmatrix} =  \frac{1}{\int_{x_l}^{x_r}\int_{y_b}^{y_t}u(t,x,y)dxdy}\int_{x_l}^{x_r}\int_{y_b}^{y_t}\begin{bmatrix} {x} \\ {y}\end{bmatrix}\cdot u(t, x, y)dydx
\end{equation} and numerically differentiating in time with centered difference. We visualize the predicted solutions for three of the DMD strategies in Figures~\ref{fig:2d-adv-sol1} and \ref{fig:2d-adv-sol2}, corresponding respectively to the standard DMD, physics-aware DMD, and time-varying DMD with Lagrangian moving grid, constructed with a subinterval size $r=30$. We predict the solutions up to $t=8$ and compare with the baseline numerical solution. Finally, the prediction errors (\ref{eqn:define-relative-error}) for all four DMD strategies are presented in Figure~\ref{fig:2d-adv-error}. Due to presence of small diffusion, a time-varying DMD strategy without Lagrangian moving grid is able to achieve comparable accuracy to that with Lagrangian information. The standard DMD shows significant degradation in accuracy over time. The physics-aware DMD and time-varying DMD with physics still possess model misspecification that results in a growth of error over time, albeit at a reduced rate than standard DMD. In contrast, the results given by Algorithm~\ref{alg:time-dependent-lagrangian-dmd} shows controlled error growth, similar to the case observed in (\ref{eqn:1d-advection}).

\newpage
\section{Conclusions}\label{sec:conclusions} In this work, we investigated a method for learning time-dependent advection-dominated phenomena using DMD algorithms. In particular, when the PDE parameters vary in time, we demonstrated that the characteristic lines of the PDE are an important observable to include in order to improve the accuracy of reconstructions, as verified with 1d and 2d advection-diffusion equations with time-varying coefficients. We further provided prediction error guarantee for the time-dependent approximation to the Koopman operator. In addition, we analyzed the effect of SVD truncation and number of data snapshots on operator norm error, and verified such upper bounds in both model-free and model-dependent cases. The method adopted in this work provides a possibility for real-time control in advection-dominated systems.

One of the possible future directions concerns the identification of closures for characterizing the time-evolution of a quantity of interest that depends on the states of another dynamical system \cite{Curtis_2021}. Instead of relying on an equation-free model, deriving and learning explicit forms of the reduced-order dynamics provides a principled analysis tool for uncertainty propagation and control system design, as well as extrapolation capabilities. Furthermore, we briefly investigated the possibility of a full data-driven model by assuming the advection coefficients are unknown and estimated by mode-tracking. Although such a method is effective in capturing the macroscopic behavior, it is far from being sufficient for velocities that have nonlinear dependence in both spatial variables and the solution itself. Future explorations will focus on parameterizations for the advection and diffusion coefficients, which are identified simultaneously as the optimal linear operator is constructed. Such a scenario can be potentially considered in a constrained optimization \cite{Ouala_2023} or Bayesian inversion setting \cite{Kawashima2023}. Reduction of computational complexity is another possible path of future exploration due to the curse of dimensionality for advection-dominated problems associated with moderate- to high-dimensional datasets. An added layer of dimensionality reduction must be adopted in such cases where storing and operating with data snapshots and the Lagrangian moving grid are intractable. A potential solution in the DMD setting is by using low-rank tensor-networks to approximate multidimensional linear operators \cite{Klus_2018, dektor2021rankadaptive}. 

\section*{Acknowledgments} We would like to thank Dr. Hannah Lu and Dr. Tyler Maltba for useful discussions and manuscript review. The research was partially supported by the Air Force Office of Scientific Research under grant FA9550-21-1-0381, by the National Science Foundation under award 2100927, by the Office of Advanced Scientific Computing Research (ASCR) within the Department of Energy Office of Science under award number DE-SC0023163, and by the Strategic Environmental Research and Development Program (SERDP)  of the Department of Defense under award RC22-3278.

\bibliographystyle{siamplain}
\bibliography{localDMD}
\end{document}